\documentclass{amsart}
\usepackage[T1]{fontenc} 

\usepackage{amssymb}
\usepackage{amsthm}
\usepackage{amsmath}
\usepackage{xcolor}
\usepackage[matrix, arrow]{xy}
\xyoption{arrow}
\xyoption{curve}
\theoremstyle{plain}\newtheorem{Theorem}{Theorem}[section]
\theoremstyle{plain}
\theoremstyle{plain}\newtheorem{Corollary}[Theorem]{Corollary}
\theoremstyle{plain}\newtheorem{Lemma}[Theorem]{Lemma}
\theoremstyle{plain}\newtheorem{Proposition}[Theorem]{Proposition}
\theoremstyle{definition}\newtheorem{Definition}[Theorem]{Definition}
\theoremstyle{definition}\newtheorem{Example}[Theorem]{Example}
\theoremstyle{definition}
\theoremstyle{definition}
\theoremstyle{definition}
\theoremstyle{definition}\newtheorem{Remark}[Theorem]{Remark}
\theoremstyle{definition}
\theoremstyle{plain}\newtheorem{Statement}[Theorem]{}
 \def\OG{{\mathcal{O}G}}

\def\CF{{\mathcal{F}}}  
\def\CG{{\mathcal{G}}}

\def\CO{{\mathcal{O}}}

\def\F{{\mathbb F}}           
\def\Fp{{\mathbb F_p}}

\def\Aut{\mathrm{Aut}}               \def\tenk{\otimes_k}     
\def\Br{\mathrm{Br}}                 \def\ten{\otimes}

 \def\tenkH{\otimes_{kH}} 

\def\dim{\mathrm{dim}}   \def\tenkP{\otimes_{kP}} \def\tenkQ{\otimes_{kQ}}

\def\End{\mathrm{End}}         
\def\Endbar{\underline{\mathrm{End}}}    
        \def\tenkR{\otimes_{kR}} \def\tenR{\otimes_{R}}

\def\Hom{\mathrm{Hom}}           \def\tenkS{\otimes_{kS}}

\def\Hombar{\underline{\mathrm{Hom}}}

\def\ker{\mathrm{ker}}           

\def\Id{\mathrm{Id}} 
      \def\tenA{\otimes_A}
      \def\tenB{\otimes_B}
\def\Ind{\mathrm{Ind}}    \def\tenC{\otimes_C}

\def\Irr{\mathrm{Irr}}           
           \def\tenO{\otimes_{\mathcal{O}}}
\def\mod{\mathrm{mod}}

\def\modbar{\underline{\mathrm{mod}}}

\def\op{\mathrm{op}}          

\def\perf{\mathrm{perf}}        
\def\perfbar{\underline{\mathrm{perf}}}

  \def\pr{\mathrm{pr}}

\def\rad{\mathrm{rad}}

\def\Res{\mathrm{Res}}

\def\soc{\mathrm{soc}}

\def\Tr{\mathrm{Tr}}

\makeindex
\title{The source permutation module of a block of a finite group algebra} 
\author{Radha Kessar}
\address{Department of Mathematics, University of Manchester,  M13 9PL
  United Kingdom}
\email{radha.kessar@manchester.ac.uk}
\author{Markus Linckelmann}
\address{Department of Mathematics, City St George's, University of 
London EC1V 0HB,
 United Kingdom}
\email{markus.linckelmann.1@city.ac.uk}

\date{July 19, 2025}
\begin{document}

\begin{abstract}
For $G$ a finite group, $k$ a field of prime characteristic $p$, and $S$ a 
Sylow $p$-subgroup of $G$,  the Sylow permutation module $\Ind^G_S(k)$ 
plays a role in  diverse facets  of  representation theory and group theory, 
ranging from Alperin's weight conjecture to statistical considerations of 
$S$-$S$-double cosets in $G$. The Sylow permutation module breaks up along the
block decomposition of the group algebra $kG$, but the resulting block components 
are not   invariant under splendid Morita equivalences. We introduce a summand of the block component, 
which we  call {\em source permutation module}, which  is shown to be invariant  under such equivalences. 
We investigate   general structural properties of the source permutation module 
 and  we show that well-known results relating the self-injectivity of
the endomorphism algebra of the Sylow permutation to Alperin's weight conjecture 
carry over to the source permutation module. We calculate this module in various 
cases, such as certain blocks with cyclic defect group and blocks with a Klein four 
defect group, and for some blocks of symmetric groups, prompted by a question 
in a recent paper by Diaconis-Giannelli-Guralnick-Law-Navarro-Sambale-Spink  
on  the  self-injectivity  of  the endomorphism algebra of  the  Sylow permutation 
module for symmetric groups.
 \end{abstract}



\maketitle

\section{Introduction}

Let $k$ be a field of prime characteristic $p$.  In \cite{Alp87}, Alperin   highlighted  the   
Sylow permutation $kG$-module $\Ind^G_S(k)=$ $kG\tenkS k$, where $G$  and
$S$ a Sylow $p$-subgroup in the context of his 
weight conjecture. By results of Sawada \cite{Sawada77}  and  Tinberg \cite{Tin80}, 
the  endomorphism algebra  of   $kG\tenkS k$
is  self-injective when $G$ has a characteristic $p$ split BN-pair, and by 
results of Green \cite{Green78},  having  a self-injective endomorphism 
algebra has strong structural consequences for a module.  The   self-injectivity  
property  was successfully  exploited  by  Cabanes   in  his proof of the Alperin 
weight  conjecture  for  finite reductive  groups  in their defining characteristic 
\cite{Ca84}.
The structure of  the endomorphism algebra of the Sylow permutation module,
with a view to its  relevance for  Alperin's weight conjecture, has   been  further  
explored  in \cite{Nae}, \cite{HKN}, \cite{BaTu}, with extensive calculations  in
 \cite{NaehrigPhD},  \cite{Nae}.    The  recent paper \cite{DGGLNSS} investigates 
 the summands   of the  restriction of  $kG\tenkS k$ to $kS$  from 
the point of view   of random walks on groups   and  raises the question of the  
self-injectivity   of  the 
endomorphism algebra  of   $kG\tenkS k$  when $G$ is a finite symmetric group. 
The  initial  motivation for  our work came    from 
an  attempt to  shed light on  this question (see Theorem~\ref{cyclic-Sp}).

\medskip
The Sylow permutation module  $kG\tenkS k$ is a  direct sum of   the  {\it block 
Sylow permutation  modules}  $B\tenkS k$, as $B$ runs through the blocks of 
$kG$.  One shortcoming   is that  these block summands  are not invariant  under   
Morita or even splendid Morita  equivalence.  We identify  a summand   of  
$B\tenkS k$  which has   such an invariance  property. 
For basic terminology and background facts  we refer to Section  
\ref{backgroundsection}. 

\begin{Definition} \label{source-module-Def}
Let $G$ be a finite group, $B$ a block of $kG$ with defect group $P$, and let
$i\in B^P$ be a source idempotent of $B$. The $B$-module $Bi\tenkP k$ is
called  the  {\em source permutation module of $B$}.
\end{Definition}

The isomorphism class of the source permutation $B$-module $Bi\tenkP k$ in
Definition \ref{source-module-Def}  does not depend on the choice of a defect 
group $P$ and source idempotent $i$ (cf.  Lemma 
\ref{isomorphic-summands-Lemma})  and it 
is invariant under splendid Morita equivalences (cf. Proposition 
\ref{splendid-stable-source}).
Our first main result is that the  source permutation module is 
 a direct summand of the   Sylow permutation module  for sufficiently large $k$.

\begin{Theorem} \label{Bi-summand} 
Let $G$ be a finite group,  $S$ a Sylow $p$-subgroup of $G$,  $B$ a block of $kG$ 
with a non-trivial defect group $P$, and let $i\in B^P$  be a source idempotent. 
Assume that $k$ is a splitting field for $iB^Pi$.  Then  $Bi\tenkP k$ is isomorphic 
to a direct  summand of   $B\tenkS k$  as  $kG$-modules. Moreover, every 
indecomposable direct summand of $B\tenkS k$ with vertex $P$ is isomorphic 
to a direct summand of $Bi\tenkP k$. 
\end{Theorem} 

Theorem~\ref{Bi-summand} is proved   as part  of 
Proposition~\ref{BiBS-summand-Prop}.

\medskip
A key property  of  the  Sylow permutation   module of a finite group $G$, 
identified by Alperin in \cite[Lemma 1]{Alp87},   is that   it   has the Green  
correspondents of   $G$-weights as direct summands.  Our second main 
result is that this property carries over to source permutation modules.  
The terminology regarding weights is  reviewed at the beginning of 
Section \ref{weight-section}.

\begin{Theorem} \label{Bi-weight-thm}
Let $G$ be a finite group,   $B$ a block of $kG$ 
with a non-trivial defect group $P$, and let $i\in B^P$  be a source idempotent.
Every Green correspondent of a $B$-weight is isomorphic to a direct summand
of the source permutation $B$-module $Bi\tenkP k$.
\end{Theorem}

This will be proved in Theorem \ref{Bi-weights}.   
 Our next result  shows that source permutation  modules   of blocks of 
 positive  defect   have no nonzero  projective  summands while at the 
 same time having all simple  modules, up to isomorphism,  in their  heads 
 and  in their  socles.

\begin{Theorem} \label{Bi-no-projective-summand-thm}
Let $G$ be a finite group,  $S$ a Sylow $p$-subgroup of $G$,  $B$ a block of $kG$ 
with a non-trivial defect group $P$, and let $i\in B^P$  be a source idempotent.
\begin{itemize}
\item[{\rm (i)}]
The source permutation $B$-module $Bi\tenkP k$ has no nonzero projective  
direct summand. 
\item[{\rm (ii)}] 
Every simple $B$-module is isomorphic to a quotient and a  submodule of  
$Bi\tenkP k$.
\end{itemize}
\end{Theorem}

The proof of  Theorem \ref{Bi-no-projective-summand-thm}   appears after  
Proposition  \ref{Bi-top-socle} (the first statement is part of    
Lemma \ref{proj-summands-Lemma} and the  second  statement   appears 
as  Proposition \ref{Bi-top-socle}).
By Theorem~\ref{Bi-no-projective-summand-thm} the source permutation module 
has no nonzero projective 
summand. It does though always have a summand with full vertex, and more
precisely,  its summands with full vertex coincide,  up to isomorphism,  with those of 
the block Sylow 
permutation  module (cf. Theorem \ref{Bi-summand}).  
There are more precise results for source permutation modules of principal blocks. 
Robinson showed  in \cite[Theorem 1.2]{Rob88} that if $Q$ is an essential 
$p$-subgroup of $G$, then $Q$ is a vertex of a summand
of the Sylow permutation module belonging to the principal block. The next
result adds that these summands appear in the source permutation module
of the principal block.

\begin{Theorem} \label{princ-block-thm}
Let $G$ be  a finite group, $S$ a Sylow $p$-subgroup of $G$, $B$ the
principal block of $kG$, and $i\in $ $B^S$ a source idempotent.
\begin{itemize}
\item[{\rm (i)}]
Every indecomposable summand of $B\tenkS k$ with a centric vertex is 
isomorphic to a direct summand of $Bi\tenkS k$. 
\item[{\rm (ii)}]
For every essential subgroup $Q$ of $G$ there is an indecomposable direct
summand of $Bi\tenkS k$ with vertex $Q$.
\end{itemize}
\end{Theorem}

This will be proved in Section \ref{princ-block-Section}, where we also review the
relevant terminology on centric and essential subgroups.

\medskip
As mentioned   at the beginning of the introduction,  one  question  of particular
interest is  whether the  endomorphism  algebra  of the Sylow permutation 
module   is self-injective.    Theorem \ref{Bi-weight-thm}  and   Theorem 
\ref{Bi-no-projective-summand-thm}   lead  to the following result.

\begin{Theorem} \label{Bi-selfinjective-thm1}
Let $G$ be a finite group,  $S$ a Sylow $p$-subgroup of $G$, 
$B$ a block of $kG$ with a non-trivial defect group $P$, and let $i\in B^P$ 
be a source idempotent.
Suppose that $k$ is a splitting field for all $B$-Brauer pairs and their normalisers.
If $\End_B(B\tenkS k)$ is self-injective, then $\End_B(Bi\tenkP k)$ is self-injective.
In that case  the two algebras are Morita equivalent, or equivalently, every 
indecomposable direct summand of $B\tenkS k$ is isomorphic to a direct 
summand of $Bi\tenkP k$.
\end{Theorem}

As a consequence of  the above results  we obtain that one deterrent to  the  
self-injectivity   of  $\End_B(B\tenkS k)$, which does not arise  for   that of  
$\End_B(Bi\tenkP k)$, is  the presence of projective   summands.

\begin{Corollary} \label{noprojective-cor} 
Let $G$ be a finite group,  $S$ a Sylow $p$-subgroup of $G$, 
$B$ a block of $kG$ with a non-trivial defect group.  Suppose that $k$ is a splitting 
field for all $B$-Brauer pairs and their normalisers. If $B\tenkS  k $ has  a  nonzero 
projective  summand, then  $\End_B(B\tenkS k)$  is not self-injective. In particular, if 
$kG \tenkS k $  has a  projective indecomposable  summand which is not simple, then  
$\End_{kG} (kG \tenkS k)$  is not self-injective. 
\end{Corollary} 

Theorem \ref{Bi-selfinjective-thm1}  and Corollary \ref{noprojective-cor} 
 will be proved at the end of Section \ref{weight-section}.  We note that 
 Theorem \ref{Bi-selfinjective-thm1} is  not a  formal result  since it is not 
  the case that  the  idempotent truncation of a self-injective algebra   is   always 
self-injective. The converse of Theorem \ref{Bi-selfinjective-thm1}  need not hold; 
counter examples include  the principal 
blocks of the symmetric group algebras $kS_p$ when $p\geq 7$ (see 
Theorem \ref{cyclic-Sp}). 

\medskip
The  interplay between   Alperin's weight conjecture and  the structure of the Sylow 
permutation module and its endomorphism algebra   remains a tantalising  theme 
within the theory. 
By a conjecture of Auslander and Reiten, a stable equivalence
between two  finite-dimensional algebras  should preserve the number of 
isomorphism classes of simple and non-projective  modules.
The following Theorem is an instance where the two conjectures meet.

\begin{Theorem}\label{splendid-thm}  
Let $G$ be a finite group, $B$ a  block of $kG$ with a non-trivial defect group   $P$, 
$ i \in B^P$ a source idempotent  and  $C$  the Brauer  correspondent  of $B$ 
in $N_G(P)$.  Assume that $k$ is a splitting field for $B$, $C$  and their Brauer pairs.
Suppose that  there is a splendid   stable  equivalence between $B$ and $C$.  
If $\End_B(Bi\tenkP k)$ is 
self-injective, then  $B$ and $C$ have the  same number of isomorphism classes  
of  simple modules and   $B$ satisfies Alperin's weight conjecture.
\end{Theorem}
  
This will be proved as part of a more general result in Section \ref{stable-Section};
see Theorem \ref{stable-Brauer-corr}.
 
\medskip
The following  two   theorems    describe situations   where  the   endomorphism 
algebra  of the source permutation module  is self-injective.

\begin{Theorem}\label{cyclic-thm}
Let $G$ be a finite group, $p$ a prime, $k$ a field of characteristic $p$,  and $B$ 
a block of $kG$ with a non-trivial cyclic defect group $P$.  
Denote by $D$ the block of $kN_G(P)$ with defect group $P$  which is the Brauer 
correspondent of $B$. Assume that $k$ is a splitting field for $B$, $D$ and their
Brauer pairs. Suppose that $B$ has an indecomposable direct summand as a 
$kG$-$kN_G(P)$-module which  induces a stable equivalence of Morita type 
between $B$ and  $D$. Let $i\in B^P$ be a source idempotent. Then the algebra
$\End_{B}(Bi\tenkP k)$ is a direct product of  self-injective Nakayama algebras. 
In particular,  $\End_{B}(Bi\tenkP k)$ is self-injective.
\end{Theorem}

This will be proved in Section \ref{cyclic-Section}, making use of general properties
of Brauer tree algebras from Section \ref{Brauertree-section}.
The hypothesis on the stable equivalence between $B$ and $D$ to be induced
by a truncated induction/restriction implies in particular that the Green 
correspondents of the simple $D$-modules are uniserial.
If $B$ is a block with a cyclic defect group and $B$  is the principal  block   
or if  $P$ has order $p$, then $B$  satisfies the hypotheses of the above theorem.   
More precisely, these hypotheses are satisfied whenever the nilpotent block 
corresponding to $B$  of the centraliser of the subgroup of 
order $p$ of $P$  has a simple module with endotrivial source; see Proposition
\ref{splendid-stable-cyclic-Prop} below.

\medskip
A  statement similar  to  that of Theorem~\ref{cyclic-thm}  holds for blocks with a 
Klein four defect group, although
this requires the classification of finite simple groups via the reference
\cite{CEKL} in which source algebras of Klein four defect blocks are classified.

\begin{Theorem} \label{Kleinfour-thm}
Suppose that $p=2$. Let $G$ be a finite group and $B$ a block of $kG$
with a Klein four defect group $P$. Assume that $k$ is a splitting field for
$B$, its Brauer correspondent and their Brauer pairs.
Let $i\in B^P$ be a source idempotent.
Then $\End_B(Bi\tenkP k)$ is isomorphic to either $k$ or $k\times k\times k$, or
$k\times N$, where $N$ is a $4$-dimensional basic Nakayama algebra with two 
isomorphism classes of simple modules and  projective indecomposable modules 
of length $2$. In particular, $\End_B(Bi\tenkP k)$ is self-injective.
\end{Theorem} 

The three cases in the above Theorem correspond to the cases where the source
algebras of $B$ are isomorphic to either $kP$, or $kA_4$, or the principal block
of $kA_5$. This will be proved in Section \ref{Kleinfour-section}. For completeness
we comment on nilpotent blocks.

\begin{Theorem} \label{nilpotent-thm}
Let $k$ be a field of prime characteristic $p$.
Let $G$ be a finite group and $B$ a nilpotent block of $kG$ with defect group $P$
and source idempotent $i\in B^P$.  Assume that $k$ is a splitting field for
$B$, its Brauer correspondent and their Brauer pairs.
Then $\End_{B}(Bi\tenkP k)\cong$ $\End_{kP}(V)^\op$, where $V$ is an 
indecomposable endopermutation $kP$-module with vertex $P$ such that
$V$ is a source of a simple $B$-module. This is a split local algebra, and if this
algebra is self-injective, then $V$ has a simple top and socle.
\end{Theorem}

This will be proved in Section \ref{nilpotent-Section}.
As  an illustration of  and motivation for  some of the   above theorems, we  
examine the  case of the  symmetric groups  $S_n$  in Section \ref{Example-section}.   
In  particular,  Theorem \ref {cyclic-Sp}  addresses the 
question  on the self-injectivity of  the endomorphism algebras 
of  Sylow permutation modules of symmetric groups raised in  
\cite[Section 6]{DGGLNSS}.
By abuse of notation, for  $Y$ a finite-dimensional module over a finite-dimensional 
algebra $A$, we  say that $M$ {\em is the non-projective part of}  $Y$, if 
$M$ is a direct summand of $Y$ such that $M$ has no nonzero projective direct 
summand and a  projective complement in $Y$. By the Krull-Schmidt Theorem, 
the isomorphism  class of $M$ is uniquely determined.

\begin{Theorem} \label{cyclic-Sp}   
Let $n$ be a positive integer, $B$  a block of  $kS_n $  with a non-trivial defect 
group $P$,   and  let $i \in B^P$   a source idempotent.  
\begin{enumerate}   

\item [{\rm (a)}]  
Assume that $P$  is cyclic. If  $p $ is  odd,  then    $\End_{B}(Bi\tenkP k)$    is a 
direct product  of two copies of $k$  and $\frac{p-3}{2}$ copies of a 
$4$-dimensional basic Nakayama  algebra with two isomorphism classes of 
simple modules and  projective  indecomposable modules of length $2$.    
If $p=2 $,  then $\End_{B}(Bi\tenkP k) =  k $.  In particular, $\End_{B}(Bi\tenkP k)$ 
is symmetric for $p=2, 3$ and   self-injective but not symmetric for $p\geq 5$.  

\item [{\rm (b)}]    
Suppose that $n=p $ and $B$ is the  principal block (so $P$ is a Sylow  $p$-subgroup 
of $S_p$).  Then $Bi\tenkP k $ is the non-projective  part of   $ B \tenkP k$. Moreover,  
$B=Bi$    if and only if  $p \leq 5 $.    Consequently,   $\End_{B}(B\tenkP k)$  and  
$\End_{kS_n}(\Ind_{P}^{S_n}(k)) $    are   self-injective if  and only  if    $p\leq 5 $.

\item [{\rm (c)}]    
Let $n >  p \geq 7 $ and let $S$  be  a Sylow $p$-subgroup of  $S_n$.  Then 
$\End_{kS_n}(\Ind^{S_n}_{S}(k))$  is not self-injective.
\end{enumerate}
\end{Theorem}

This will be proved in Section \ref{Example-section}.
We note that   the  the  last statement of part (b)  for  $p\leq  7 $     is  proved by  
Naehrig in   \cite{NaehrigPhD}.   In the same work,  Naehrig   also  shows  that    
when $p=3$, then  $\End_{kS_6}(\Ind_S^{S_6} (k))$  is self-injective.

\begin{Remark}
The $p$-permutation modules considered in this paper all lift to a complete
discrete valuation ring with $k$ as residue field, as do many of the results.
We comment on this in Remark \ref{O-Remark} below.
\end{Remark}

\section{Background on algebras and blocks of finite groups} \label{backgroundsection} 

Let $k$ be a field.  Given two $k$-algebras $A$, $B$, an $A$-$B$-bimodule $M$ is
understood to have the same left and right $k$-module structure, or equivalently,
$M$ is a left $A\tenk B^\op$-module,with $a\ten b$ acting on $m\in M$ as $amb$,
where $B^\op$ is the opposite algebra of $B$, and where $a\in A$ and $b\in B$. 
Thus an $A$-$B$-bimodule is projective if it is projective as an  
$A\tenk B^\op$-module. For $A$ a finite-dimensional $k$-algebra, we denote by
$\ell(A)$ the number of isomorphism classes of simple $A$-modules and by
$J(A)$ the Jacobson radical of $A$.
A finite-dimensional $k$-algebra $A$ is called {\em semisimple} if $J(A)=0$, and
{\em separable} if  $J(k'\tenk A)=0$ for any extension field $k'$ of $k$.
Given an arbitrary finite-dimensional $k$-algebra, the quotient $A/J(A)$ is
semisimple, and this quotient is separable if $J(k'\tenk A)=$ $k'\tenk J(A)$ for
any extension field $k'$ of $k$; in this case we say for short that $A$ has a
separable semisimple quotient. 
We will  make repeated  use of  the following fact.

\begin{Proposition}[{\cite[Proposition 4.14.9]{LiBookI}, 
\cite[Proposition 2.3]{Listable}}]  \label{no-proj-summands-Prop}
Let $A$, $B$ be finite-dimensional self-injective $k$-algebras with
separable semisimple quotients. Let $M$ be an $A$-$B$-bimodule
which has no nonzero projective direct summand as an $A\tenk B^\op$-module
and which is finitely generated projective as a right $B$-module. 
Then, for any non-projective simple $B$-module $T$, the $A$-module $M\tenB T$
has no nonzero projective direct summand.
\end{Proposition}

It is well-known that finite group algebras, hence their blocks and source algebras
(see below), are symmetric with separable semisimple quotients. 
A finite-dimensional $k$-algebra $A$ is called {\em split} if $\End_A(S)\cong$ 
$k$ for any simple $A$-module $S$.  An extension field $k'$ of $k$ is called a
{\em splitting field for} $A$ if the $k'$-algebra $k'\tenk A$ is split.
If $A$ is split, then every extension field of $k$ is a splitting field for $A$.
If $A$ is split, then $A/J(A)$ is split semisimple, hence a direct product of
matrix algebras over $k$, and in particular, $A/J(A)$ is separable. 

\begin{Lemma} \label{ij-primitive-Lemma}
Let $A$, $B$ be finite-dimensional $k$-algebras. Let $i\in A$ and $j\in B$ be
primitive idempotents. If one of $A$ or $B$ is split, then $i\ten j$ is a primitive
idempotent in $A\tenk B$. 
\end{Lemma}

\begin{proof} 
Suppose that $A$ is split. Then $iAi$ is split local; that is, we have 
$iAi/iJ(A)i\cong$ $k$. Thus $J(iAi\tenk jBj)=$ 
$J(iAi)\tenk jBj \ + iAi\tenk jJ(B)j$ (see e.g.  \cite[Corollary 1.16.15]{LiBookI}). 
It follows that $iAi\tenk jBj / (J(iAi\tenk jBj)\cong$  $ jBj/jJ(B)j$ is local, hence 
$i\ten j$ is primitive by standard lifting theorems.
\end{proof}

For $A$ a finite-dimensional $k$-algebra, we denote by $\mod(A)$ the category 
of finitely generated left $A$-modules, and by $\modbar(A)$ its stable category.
That is, $\modbar(A)$ has the same objects as $\mod(A)$, and for any two
finitely generated $A$-modules $U$, $V$ the morphism space from $U$ to $V$
in $\modbar(A)$ is the vector space $\Hombar_A(U,V)=$ 
$\Hom_A(U,V)/\Hom_A^\pr(U,V)$, where $\Hom_A^\pr(U,V)$  is the space of
$A$-homomorphisms from $U$ to $V$ which factor through a projective $A$-module.
The composition of morphisms in $\modbar(A)$ is induced by the usual composition
of $A$-homomorphisms.
If $A$ is a self-injective algebra, then $\modbar(A)$ is a triangulated category,
with exact triangles induced by short exact sequences in $\mod(A)$. 

For $A$, $B$  finite-dimensional
$k$-algebras,  an $A$-$B$-bimodule $M$ is called {\em perfect} if $M$
is finitely generated projective as a left $A$-module and as a right $B$-module.
We denote by $\perf(A,B)$ the category of perfect $A$-$B$-bimodules; this is a
full additive subcategory, closed under taking direct summands, 
of the category $\mod(A\tenk B^\op)$ of
finitely generated $A\tenk B^\op$-modules. We denote by $\perfbar(A,B)$ the
corresponding stable category of $\perf(A,B)$; that is, $\perfbar(A,B)$ is the
full subcategory of $\modbar(A\tenk B^\op)$ of all perfect $A$-$B$-bimodules.
If $A$ and $B$ are self-injective, then $A\tenk B^\op$ is self-injective, and
$\perfbar(A,B)$ is thick subcategory
of the triangulated category $\modbar(A\tenk B^\op)$.
The following is well-known.

\begin{Lemma}[{cf. \cite[Propositions 9.1.1, 9.1.2]{LiBookII}}] \label{perf-Lemma}
Let $A$, $B$, $C$  be finite-dimensional $k$-algebras.  
If $M$ is a perfect $A$-$B$-bimodule and $U$ a perfect $B$-$C$-bimodule, 
then $M\tenB U$ is a perfect $A$-$C$-bimodule, and moreover,  if one of  
$M$, $U$ is a projective bimodule, then so is their tensor product  $M\tenB U$.
\end{Lemma} 

For the remainder of this section, we assume that $k$ is a field of prime  characteristic $p$.

\medskip
Let $G$ be a finite group.
We assume familiarity with Green's theory of vertices and sources of 
indecomposable $kG$-modules (see e.g. \cite[Section 5.1]{LiBookI}),
the Green correspondence (\cite[Section 5.2]{LiBookI}),
and we refer to \cite[Sections 5.10, 5.11]{LiBookII} for the notions of trivial 
source modules and $p$-permutation modules. Given two finite groups
$G$, $H$, a $kG$-$kH$-bimodule $M$ can be viewed as a $k(G\times H)$-module
with the action $(x,y)\cdot m=$ $xmy^{-1}$, where $x\in G$, $y\in H$, and $m\in M$.
If $M$ is indecomposable, its vertices are understood as the vertices of $M$
regarded as a $k(G\times H)$-module.
For the translation between $kG$-$kH$-bimodules and $k(G\times H)$-modules
we will need the following well-known observation.

\begin{Lemma} \label{bimod-Lemma}
Let $G$, $H$ be finite group, $P$ a $p$-subgroup of $G$, and $\varphi : P\to H$
an injective group homomorphism. Set $R=$ $\{(u,\varphi(u))\ |\ u\in P\}$.
We have an isomorphism of $k(G\times H)$-modules
$$\Ind^{G\times H}_R(k)\cong kG\tenkP ({_\varphi{kH}})$$ 
sending $(g,h) \ten 1$ to $g \ten h^{-1}$,  where $g\in G$, $h\in H$, and where
${_\varphi{kH}}$ is the $kP$-$kH$-bimodule equal to $kH$ as a right
$kH$-module, with $u\in P$ acting by left multiplication with $\varphi(u)$.
\end{Lemma}

\begin{proof} This is a straightforward verification.
\end{proof}

We will need the following well-known facts on tensoring perfect
$p$-permutation bimodules.
For a much more general formula see Bouc \cite[Theorem 1.1]{Bouc10}.

\begin{Lemma} \label{perfect-vertices-Lemma}
Let $G$, $H$, $L$ be finite groups, let $M$ be a perfect indecomposable 
trivial source $kG$-$kH$-bimodule and $N$ a perfect indecomposable
trivial source $kH$-$kL$-bimodule. The following hold.
\begin{itemize}
\item[{\rm (i)}]
Any vertex of $M$ as a $k(G\times H)$-module is of the form 
$\Delta_\varphi(P)=$ $\{(u,\varphi(u))\ |\ u\in P\}$ for some $p$-subgroups
$P$, $Q$ of $G$, $H$ respectively, and some isomorphism $\varphi : P\to Q$.
\item[{\rm (ii)}]
The $kG$-$kL$-bimodule $M\tenkH N$  is perfect, isomorphic to a direct sum 
of trivial source  modules, and any vertex of an indecomposable direct summand 
of $M\tenkH N$  has order less or equal than the orders of the vertices of $M$ 
and of $N$.
\end{itemize}
\end{Lemma}

\begin{proof} See e.g. \cite[Theorem 5.1.16]{LiBookI}.
\end{proof}

\medskip
We review very briefly the
notion of pointed groups due to Puig \cite{Puigpoint}; for more details we refer to
 \cite[Sections 5.5, 5.6]{LiBookI} or \cite{Thev}. 
 Let $G$ be a finite group.
A $G$-{\em algebra} is a $k$-algebra $A$ endowed with an action of $G$ by
algebra automorphisms, denoted $a\mapsto {^x{a}}$, where $a\in$ $A$ and $x\in G$.
For $H$ a subgroup of $G$ we denote by $A^H$ the subalgebra of fixed points in $A$
under the action of $H$.  Following Puig \cite{Puigpoint}, a {\em point of $H$ on $A$}
is an $(A^H)^\times$-conjugacy class $\beta$ of primitive idempotents in $A^H$,
and then $H_\beta$ is called a {\em pointed group on $A$}. Given another pointed
group $L_\gamma$ on $A$ we write $L_\gamma\leq H_\beta$ if $L\leq H$ and if
there are idempotents $j\in$ $\gamma$ and $i\in \beta$ such that $ji=j=ij$.  Conjugation 
by $G$ on the set of subgroups of $G$ and the action of $G$ on $A$ induce an action of 
$G$ on the partially ordered set of pointed  groups on $A$. If $L$ is a subgroup of $H$, 
then $A^H\subseteq$ $A^L$,  and we  have a {\em relative trace map} 
$\Tr^H_L : A^L\to$ $A^H$ sending $a\in A^L$ to  $\Tr^H_L(a)=$ $\sum_{x\in [H/L]} {^x{a}}$, 
where $[H/L]$ is a set of representatives  in $H$ of $H/L$;  one checks that this is 
well-defined (see e.g.  \cite[Definition 2.5.1]{LiBookI}). We denote by $A^H_L$ the inage 
of the relatve trace map $\Tr^H_L$; this is an ideal in the algebra $A^H$. If $P$ is a 
$p$-subgroup of $G$, we denote by $A(P)$ the quotient $A^P/\sum_{Q; Q<P}\ A^P_Q$
and by $\Br_P^A : A^P\to$ $A(P)$ the canonical surjection. A point $\gamma$ of $P$ on
$A$ is called {\em local} if $\Br_P^A(\gamma)\neq 0$; in that case, $P_\gamma$ is
called a {\em local pointed group on $A$}, and by standard lifting theorems for
idempotents, $\Br^A_P(\gamma)$ is a conjugacy class of primitive idempotents 
in $A(P)$, hence corresponds to a unqiue isomorphism class of simpe $A(P)$-modules
and of projective indecomposable $A(P)$-modules.
If $A=kG$, regarded as a $G$-algebra with $G$ acting by conjugation,
then $A(P)$ can be identified canonically with $kC_G(P)$ and $\Br_P^{kG}$ 
becomes with this identification  the unique linear map $(kG)^P\to$ $kC_G(P)$ which 
is the identity on $kC_G(P)$ and which sends non-trivial $P$-conjugacy class sums to zero.  
A {\em defect group of a pointed group $H_\beta$ on $A$} is  a minimal subgroup 
$Q$ of $H$ subject to $\beta\subseteq$ $\Tr^H_Q(A^Q)$;  the defect groups of 
$H_\beta$ form an $H$-conjugacy class of $p$-subgroups of $H$. If $Q$ is a defect 
group of $H_\beta$, then there is a point $\delta$ of $Q$ on $A$ such that
$Q_\delta\leq$ $H_\beta$ and such that $\beta\subseteq$ $\Tr^H_Q(A^Q\delta A^Q)$.
In that case $\delta$ is a local point of $Q$ on $A$, and we say that $Q_\delta$ is
a {\em defect pointed group of $H_\beta$}. The defect pointed groups of $H_\beta$
form a single $H$-conjugacy class of local pointed groups on $A$, and a pointed
group $Q_\delta$ is a defect pointed group of $H_\beta$ if and only if
$Q_\delta$ is maximal subject to being local and satisfying $Q_\delta\leq$ $H_\beta$.
We will make use of the following well-known translation between idempotents
and bimodule summands of $kG$.

\begin{Lemma} \label{bimod-idempotent-Lemma}
Let $G$ be a finite group, $H$ a subgroup, and $e$ an idempotent in $(kG)^H$.
The map sending $a \in$ $(ekGe)^H$ to right multiplication by $a$ on $kGe$
is an algebra homomorphism $(ekGe)^H\cong$ $\End_{G\times H}(kGe)^\op$, 
and the following hold.
\begin{itemize}
\item[{\rm (i)}]
Any direct summand of $kGe $  as a  $k(G \times H)$-module  is  equal to $kGf $  
for some idempotent  $f\in (ekGe)^H $.
In patricular, the $kG$-$kH$-bimodule $kGe$ is indecomposable if and only if 
$e$ is primitive in  $(kG)^H$.
\item[{\rm (ii)}]
The $kG$-$kH$-bimodule $kGe$ is absolutely indecomposable if and only if 
$(ekGe)^H$ is  split local.
\item[{\rm (iii)}]
The $kG$-$kH$-bimodule $kGe$ is projective if and only if $e\in $ $(kG)^H_1$.
\item[{\rm (iv)}]
Suppose that the idempotent $e$ is primitive in $(kG)^H$. Denote by $\beta$ 
the point of $H$ on $kG$ containing $e$. A $p$-subgroup $P$ of $H$  is a defect group 
of $H_\beta$ if and only if $\Delta P=$ $\{(u,u)\ |\ u\in P\}$ is a vertex of $kGe$ 
regarded as a $k(G\times H)$-module. 
\end{itemize}
\end{Lemma}

\begin{proof}
See  e.g. \cite[Lemma 5.12.7]{LiBookI} for a proof of the first statement as well
as the statements (i) and (ii).  Note that   \cite[Lemma 5.12.7]{LiBookI}   
addresses summands of $kG$   but the proof is identical for  summands of $kGe $.
Statement (iii) follows from Higman's criterion 
\cite[Theorem 2.6.2]{LiBookI}, and it also follows from (iv). In order to prove (iv), 
note first that $kG\cong$ $\Ind^{G\times G}_{\Delta G}(k)$ as $k(G\times G)$-module
(cf. \cite[Corollary 2.4.5]{LiBookI}).  Thus  $kGe$ is a summand of $kG\cong$ 
$\Res^{G\times G}_{G\times H}(\Ind^{G\times G}_{\Delta G}(k))\cong$
$\Ind_{\Delta H}^{G\times H}(k)$ (cf. \cite[Proposition 2.4.14]{LiBookI}).
This shows that every  indecomposable summand of this
$k(G\times H)$-module has a vertex contained in $\Delta H$.
Since $kGe$, regarded as a $k(G\times H)$-module,  is a trivial source module, 
it follows  from  \cite[Proposition 5.10.3]{LiBookI} that a $p$-subgroup $P$
of $H$ has the property that $\Delta P$ is contained in a vertex of $kGe$
if and only if $\Br_P(e)\neq 0$.  Thus $\Delta P$ is a vertex of $kGe$ if and
only if $P$ is maximal subject to $\Br_P(e)\neq 0$, and the latter condition
is equivalent to $P$ being a defect group of $H_\beta$.
\end{proof} 

Note that $kG$-$kH$-bimodules of the form $kGe$ as in Lemma 
\ref{bimod-idempotent-Lemma} are perfect, and as noted above they are 
$p$-permutation  modules, when regarded as  $k(G\times H)$-modules.
By results of Brou\'e in \cite{Broue85}, the Brauer construction applied to
$p$-permutation modules commutes with taking homomorphism spaces. 
Given a finite group $G$ and $kG$-modules $U$, $V$, we regard $\Hom_k(U,V)$
as a $kG$-module via the action of $x\in G$ on $\varphi\in$ $\Hom_k(U,V)$
given by ${(^x{\varphi})}(m)=$ $x\varphi(x^{-1}m)$, where $m\in M$.
Note that for any subgroup $H$ of $G$ we have $\Hom_k(U,V)^H=$
$\Hom_{kH}(U,V)$.

\begin{Lemma}[{cf. \cite[(3.3)]{Broue85}}] \label{Brauer-p-permutation}
Let $G$ be a finite group and $P$ a $p$-subgroup of $G$.
Let $U$, $V$ be finitely generated $p$-permutation $kG$-modules.
We have a canonical isomorphism of $kN_G(P)$-modules 
$$(\Hom_k(U,V))(P)\cong \Hom_k(U(P), V(P)).$$
\end{Lemma} 

\begin{Lemma} \label{Hom-p-permutation}
Let $G$ be a finite group, $P$ a $p$-subgroup of $G$, and let $M$, $N$ be finitely 
generated $kG$-modules. Denote by $\Br_P$ the Brauer homomorphism
$\Hom_{kP}(M,N)\to$ $(\Hom_k(M,N))(P)$. Suppose that $M$ is relatively 
$kP$-projective. Then
$$\Hom_{kG}(M,N) \cap \ker(\Br_P)\ =\ \sum_{Q<P}\ \Hom_k(M,N)_Q^G,$$
where on the sight side the sum is taken over all proper subgroups $Q$ of $P$.
\end{Lemma}

\begin{proof}
Let $\varphi\in$ $\Hom_{kG}(M,N)\subseteq$ $\Hom_{kP}(M,N)$ such that 
$\varphi$ belongs to the kernel $\ker(\Br_P)$ of the Brauer construction 
applied to $\Hom_k(M,N)$ with $P$ acting as described above. That is, we have
$$\varphi = \sum_{Q<P}\ \Tr^P_Q(\psi_Q), $$
where $\psi_Q\in$ $\Hom_{kQ}(M,N)$ for each proper subgroup $Q$ of $P$
and where 
$$\Tr^P_Q(\psi_Q)(m)= \sum_{u\in P/Q}\ u\psi_Q(u^{-1}m)$$
for all $m\in $ $M$. Since $M$ is relatively $kP$-projective, it follows from
Higman's criterion (see e.g. \cite[Theorem 2.6.2]{LiBookI}) that $\Id_M=$
$\Tr^G_P(\mu)$ for some $kP$-endomorphism $\mu$ of $M$.
Thus
$$\varphi = \varphi\circ\Id_M = \varphi \circ \Tr^G_P(\mu) =
\Tr^G_P(\varphi\circ\mu) = $$
$$\sum_{Q<P}\ \Tr^G_P(\Tr^P_Q(\psi_Q)\circ\mu) =
\sum_{Q<P}\ \Tr^G_Q(\psi_Q\circ\mu).$$
This shows that the left side is contained in the right side. The other
inclusion is obvious.
\end{proof}

A {\em block of} $kG$ is an indecomposable direct factor $B$ of
$kG$ as an algebra. Any such block is equal to $kGb$ for a primitive idempotent 
$b$ in $Z(kG)$. We will say that $b$ is the block idempotent associated with the 
block $B$; this is the unit element $1_B$ of $B$.
A {\em defect group} of a block $B$ of $kG$  is a maximal $p$-subgroup $P$ of 
$G$ such that $kP$ is isomorphic to a direct summand of $B$ as a 
$kP$-$kP$-bimodule.  Equivalently, a defect group of $B$ is a minimal 
$p$-subgroup $P$ of $G$ such that $B$ is isomorphic to a direct summand of 
$B\tenkP B$ as a $B$-$B$-bimodule, and clearly $B$ is perfect as a 
$B$-$B$-bimodule. 
See \cite[Theorem 6.2.1]{LiBookII} for further  characterisations of
defect groups. A {\em source idempotent} in the fixed point algebra $B^P$ under 
the conjugation action of $P$ is a primitive idempotent $i$ in $B^P$ such that 
$kP$ is still isomorphic to a direct summand of $iBi$ as a $kP$-$kP$-bimodule. 
This is equivalent to stating that $B$ is isomorphic to a direct summand of 
$Bi\tenkP iB$ as a  $B$-$B$-bimodule; this follows from combining 
\cite[Theorem 6.4.6, Theorem 6.4.7]{LiBookII}.   Viewing  $kG$ as a $G$-algebra  
via conjugation,  the points of $G$ on $kG$ correspond bijectively to the block
idempotents $b$ of $kG$, and if $P_\gamma$ is  a defect pointed group of 
$G_{\{b\}}$, then $P$ is a defect group of the block $B=kGb$ and $i\in \gamma$ is a
source idempotent of $B$.  Equivalently,  $P$ is a maximal $p$-subgroup
such that $\Br_P(b)\neq 0$, and $i$ is a primitive  idempotent in $B^P$ such that
$\Br_P(i)\neq 0$.   The following facts are
well-known.

\begin{Lemma}\label{source-bimod-Lemma}
Let $G$ be a finite group, $B$ a block of $kG$ with a non-trivial defect group,
and $i\in B^P$ a source idempotent.
\begin{itemize}
\item[{\rm (i)}]
As a $k(G\times G)$-module, $B$ is a trivial  source module with vertex 
$\Delta(P)=$ $\{(u,u)\ | \ u\in P\}$.
\item[{\rm (ii)}]
As a $k(G\times P)$-module, $Bi$ is a trivial  source module with vertex 
$\Delta(P)=$ $\{(u,u)\ | \ u\in P\}$.
\item[{\rm (iii)}]
As a $k(P\times P)$-module, $iBi$ is a permutation module, every 
indecomposable direct summand of $iBi$ as a $k(P\times P)$-module
has a vertex of the form $\Delta_\varphi(Q)=$ $\{(u,\varphi(u))\ |\ u\in Q\}$
for some subgroup $Q$ of $P$ and some injective group homomorphism
$\varphi : Q\to P$, and $iBi$ has an indecomposable direct summand
isomorphic to $kP$, hence with vertex $\Delta(P)$.
\end{itemize}
\end{Lemma}

\begin{proof}
See e.g. \cite[Theorems 6.2.1,  6.15.1,  8.7.1]{LiBookII} for more general statements.
\end{proof}

 We need the following 
observation from \cite{LiKleinfour}  (where $k$ is assumed algebraically closed, 
but one checks that the proof given there  does not use this hypothesis; see 
\cite[Theorem 6.4.10]{LiBookII}).

\begin{Proposition}[{cf. \cite[Proposition 6.3]{LiKleinfour}}] \label{iU-vertices}
Let $G$ be a finite group, $B$ a block of $kG$, $P$ a defect group of $B$ and
$i\in B^P$ a source idempotent. Let $U$ be an indecomposable $B$-module.
Then $U$ has a vertex-source-pair $(Q,V)$ such that $Q\leq P$ and such
that $V$ is isomorphic to a direct summand of $\Res_Q(iU)$.
\end{Proposition}

\begin{Corollary}\label{iU-permutation-summands}
Let $G$ be a finite group, $B$ a block of $kG$, $P$ a defect group of $B$ and
$i\in B^P$ a source idempotent. Let $U$ be an indecomposable trivial source 
$B$-module. Then $iU$, as a $kP$-module, has a direct summand isomorphic
to $kP/Q$ for some vertex $Q$ of $U$.
\end{Corollary}

\begin{proof}
By  Proposition  \ref{iU-vertices} there is a vertex $Q$ of $U$ such that $Q\leq P$ and
such that $\Res_{Q}(iU)$ has a trivial direct summand.
Since $U$ is a direct summand of $\Ind^G_Q(k)$, it follows from the Mackey formula
that every indecomposable direct summand of $\Res^G_P(U)$, hence of $\Res_P(iU)$,
is isomorphic to $kP/R$ for some subgroup $R=$ $P\cap {^x{Q}}$, for some
$x\in G$. Again by the Mackey formula, for any such $R$ the $kP$-module 
$kP/R$ restricted to $kQ$ has
a trivial $kQ$-summand if and only if $Q$ is $P$-conjugate to a subgroup of  
$R$, hence to $R$. The Lemma follows.
\end{proof}

We will use without further comment the following standard Morita equivalence.

\begin{Proposition} [{\cite[3.5]{Puigpoint}, cf.\cite[Theorem 6.4.6]{LiBookII}}]
\label{BiMorita}
Let $G$ be a finite group, $B$ a block of $kG$, $P$ a defect group of $B$ and
$i$ an idempotent in $B^P$ satisfying $\Br_P(i)\neq 0$.
Then the algebras $B$ and $iBi$ are Morita equivalent; more precisely, there
is a Morita equivalence given by the $B$-$iBi$-bimodule $Bi$ and the
$iBi$-$B$-bimodule $iB$.
\end{Proposition}

\section{On the source permutation module, and proofs of Theorems
\ref{Bi-summand} and \ref{Bi-no-projective-summand-thm} } \label{Sylow-section}

Let $k$ be a field of prime characteristic $p$.
Let $G$ be a finite group and let $S$ be a Sylow $p$-subgroup of $G$, $B$ a block
of $kG$ with defect group $P$, and $i\in B^P$ a source idempotent. Recall
from Definition \ref{source-module-Def} that we call $Bi\tenkP k$  source
permutation module of $B$. We note first that the isomorphism class of 
$Bi\tenkP k$ does not depend on the choice of $P$ and $i$. 

\begin{Lemma} \label{isomorphic-summands-Lemma}
Let $G$ be a finite group, and let $Q_\delta$, $R_\epsilon$ pointed groups on 
$kG$.  Let $i\in \delta$ and $j\in \epsilon$. 
\begin{itemize}
\item[{\rm (i)}] 
If $Q$, $R$ are $G$-conjugate, then $kG\tenkQ k\cong$ $kG\tenkR k$  as $kG$-modules.
\item[{\rm (ii)}]
If $Q_\delta$ and $R_\epsilon$ are  $G$-conjugate, then $kGi\tenkQ k\cong$ 
$kGj\tenkR k$ as $kG$-modules.
\end{itemize}
In particular, the isomorphism class of the source permutation module $Bi\tenkP$ 
of a block $B$ of $kG$ with defect group $P$ and source idempotent $i\in B^P$ 
does not depend on the choice of $P$ and $i$.
\end{Lemma}

\begin{proof}
Let $x\in G$ such that $Q^x=R$. One checks that the map sending $a\in kG$ to 
$ax$ induces an isomorphism as stated in (i).
Let now $x\in$ $G$ such that $(Q_\delta)^x=$ $R_\epsilon$. Then  $i^x\in$
$\epsilon$, and hence  there is $c\in ((kG)^R)^\times$ such that $j=i^{xc}$. 
One checks that the map sending $ai$ to $aixc$, where $a\in kG$, induces an
isomorphism as stated in (ii). Since the defect local pointed groups of a block
$B$ of $kG$ are all $G$-conjugate, the last statement follows from (ii). 
\end{proof}

The  canonical  Morita equivalence between $B$ and its source  algebra $A=iBi$   
yields   an  algebra   isomorphism
\begin{Statement} \label{EndBi-isom}
$$\End_{B}(Bi\tenkP k) \cong \End_A(A\tenkP k).$$
\end{Statement}
From this it follows that  the endomorphism algebras of  source permutation  
modules are invariant under  source  algebra equivalences. By contrast, 
neither $\End_{B}(B\tenkP k)$ nor $\End_B(B\ten_{kS} k)$ 
are source algebra invariants. 

In  Section \ref{stable-Section},  we will  generalise  the  above   to showing  
the invariance of the source permutation module under splendid stable 
equivalences of    Morita  type.
The  next  sequence of results leads  up to  the proof of Theorem \ref{Bi-summand}.

\begin{Lemma} \label{vertices-Lemma}
Let $G$ be a finite group and $P_\gamma$  be a pointed group on $kG$.
Let $i\in \gamma$. Then every indecomposable direct summand of
the $kG$-module $kGi\tenkP k$ has a vertex contained in $P$. If $P_{\gamma} $  
is  local, then  there is at least one such summand with vertex $P$.
\end{Lemma}

\begin{proof}
Since $kGi\tenkP k$ is a summand of $\Ind^G_P(k)$, it follows that every 
indecomposable direct summand
of $kGi\tenkP k$ has a vertex contained in $P$.  
Suppose now that  $P_{\gamma} $ is local.  Since $\Br_P(i)\neq 0$, it
follows (e.g. from \cite[Lemma 5.8.8]{LiBookI}) 
that $kGi$ has a direct summand isomorphic to $kP$ as a $kP$-$kP$-bimodule,
and hence $kGi\tenkP k$ has a trivial direct summand as a left $kP$-module.
Thus some indecomposable direct summand $U$ of $kGi\tenkP k$ has  the
property that $\Res^G_P(U)$ has a trivial direct summand, hence that $P$ is
contained in a vertex of $U$. The result follows.
\end{proof}

In what follows, the hypothesis that $k$ is a splitting field for $jkG^Qj$ is
equivalent to requiring that $kGj$ is absolutely indecomposable as a 
$kG$-$kQ$-bimodule, where $j$ is a primitive idempotent in $kG^Q$. 
This hypothesis is needed in some of the results below in
order to be able to apply  Green's Indecomposability Theorem 
\cite[Theorem 5.12.3]{LiBookI}  as well as  \cite[Theorem 5.12.8]{LiBookI}.

\begin{Lemma} \label{kGi-summands-R}
Let $G$ be a finite group, $P_\gamma$ a pointed group on $kG$, and $Q$ a 
subgroup of $P$ such that $\gamma\subseteq$ $(kG)^P_Q$. Let $i\in \gamma$.
\begin{itemize}
\item[{\rm (i)}]
The $kG$-$kP$-bimodule $kGi$ is isomorphic to a direct summand of 
$kGj\tenkQ kP$ for some primitive idempotent $j\in$ $(ikGi)^Q$, and then
the $kG$-module $kGi\tenkP k$ is isomorphic to a direct summand of $kGj\tenkQ k$.
\item[{\rm (ii)}]
If $k$ is a splitting field for $(jkGj)^Q$, then  $kGj$ is absolutely indecomposable
as a $kG$-$kQ$-bimodule, $kGi$ is absolutely indecomposable
as a $kG$-$kP$-bimodule, we have an isomorphism of $kG$-$kP$-bimodules 
$kGi\cong kGj\tenkQ kP$ and an isomorphism of $kG$-modules
$kGi\tenkP k\cong$ $kGj\tenkQ k$.
\end{itemize}
\end{Lemma}

\begin{proof} 
The argument is from the proof of \cite[Proposition 6.3]{LiKleinfour}; see
also \cite[Theorem 5.12.8]{LiBookI}. By   Lemma \ref{bimod-idempotent-Lemma},  
the $kG$-$kP$-bimodule  $kGi$ is  indecomposable. By the assumptions, we have  
$i=\Tr^P_Q(d)$ for some $d\in (ikGi)^Q$. Right multiplication by $d$ on $kGi$ is a 
$k(G\times Q)$-endomorphism $\varphi$ of $kGi$ with the property that
$\Tr^{G\times P}_{G\times Q}(\varphi)$ is equal to right multiplication by 
$\Tr^P_Q(d)=$ $i$, hence equal to $\Id_{kGi}$. Thus, by Higman's criterion 
(cf. \cite[Theorem 2.6.2]{LiBookI}), $kGi$ is relatively $k(G\times Q)$-projective,
hence a direct summand of $X\tenkQ kP$ for some indecomposable direct
summand $X$ of $kGi$ as a $k(G\times Q)$-module. By Lemma
\ref{bimod-idempotent-Lemma}, $X\cong kGj$ for some primitive
idempotent $j\in$ $(ikGi)^Q$. This shows the first statement in  (i), and
tensoring the $kG$-$kP$-modules in this statement by $-\tenkP k$ yields
the remaining part of (i). 
If $k$ is a splitting field for $(jkGj)^Q$, then $(jkGj)^Q$ is split local,
hence $kGj$ is absolutely indecomposable as a $kG$-$kQ$-bimodule.
By Green's Indecomposability Theorem (cf. \cite[Theorem 5.12.3]{LiBookI}),
the $k(G\times P)$-module $kGj\tenkQ kP$ is absolutely indecomposable,
hence isomorphic to $kGi$. Tensoring this isomorphism  by $-\tenkP k$  yields
the last isomorphism in (ii).
\end{proof}

The following Proposition collects some basic facts on pairs of pointed groups 
$P_\gamma$, $Q_\delta$ on a finite group algebra $kG$ such that $Q_\delta$ 
is a defect pointed  group of $P_\gamma$. An in-depth investigation of this 
situation can be found in  Barker \cite{Barker25}.

\begin{Proposition} \label{kGi-summands-Prop}
Let $G$ be a finite group, and let $P_\gamma$, $Q_\delta$ be pointed groups
on $kG$ such that $Q_\delta\leq P_\gamma$ and such that $Q_\delta$ is
a defect pointed group of $P_\gamma$. Let $i\in \gamma$ and $j\in \delta$.
Suppose that $k$ is a splitting field for  $(jkGj)^Q$.
\begin{itemize}
\item[{\rm (i)}]
We have  an isomorphism of $kG$-$kP$-bimodules $kGi\cong$ $kGj\tenkQ kP$.
\item[{\rm (ii)}] 
The $kG$-$kP$-bimodule $kGi$ is absolutely indecomposable.
\item[{\rm (iii)}] 
We have an isomorphism of $kG$-modules $kGi\tenkP k\cong$ $kGj\tenkQ k$.
\item[{\rm (iv)}]
The $kG$-module $kGi\tenkP k$ has an indecomposable direct summand with
vertex $Q$.
\end{itemize}
\end{Proposition}

\begin{proof}
Since $Q_\delta$ is a defect local pointed group of $P_\gamma$, we have
in particular $\gamma\subseteq$ $(kG)^P_Q$.  Lemma \ref{kGi-summands-R}
implies that $kGi$ is isomorphic to a direct summand of $kGj'\tenkQ  kP$  for
some primitive idempotent $j'$ in $(ikGi)^Q$.
The fact that $Q_\delta$ is a defect pointed group of $P_\gamma$ implies that 
$j'$ belongs to a local point of $Q$. The defect local
pointed groups of $P_\gamma$ are $P$-conjugate, so we may choose $j'=j$.
The statements (i), (ii), (iii)  follow from Lemma \ref{kGi-summands-R},
and statement (iv) follows from (iii) and  Lemma \ref{vertices-Lemma}.
\end{proof}

\begin{Corollary} \label{vertices-on-BS}
Let $G$ be a finite group, $B$ a block of $kG$, let $Q_\delta$ be a local pointed
group on $B$, and let $S$ be a Sylow $p$-subgroup of $G$ containing $Q$.
Suppose that $k$ is a splitting field for $jB^Qj$, where $j\in \delta$.
If $Q_\delta$ is a defect local pointed group of $S_\mu$ for some point $\mu$ 
of $S$ on $B$, then $B\tenkS k$ has an indecomposable direct summand with 
vertex $Q$.
\end{Corollary}

\begin{proof}
Note that if $i\in$ $\mu$, then $kGi=$ $Bi$, and $Bi\tenkS k$ is a summand 
of $B\tenkS k$. Thus the statement follows from Proposition 
\ref{kGi-summands-Prop} (iv) applied with $S$, $Q$  instead of $P$, $Q$.
\end{proof}

\begin{Proposition} \label{BiBS-summand-Prop}
Let $G$ be a finite group and $B$ a block
of $kG$.  Let $P_\gamma$ be a defect local pointed group on $B$,
let $i\in\gamma$,  and let $S$ be a Sylow $p$-subgroup of $G$ containing $P$.
There exists a point $\nu$ of $S$  such that $P_\gamma\leq S_\nu$, and then
$P_\gamma$ is a defect local pointed group of $S_\nu$. Moreover, if
$k$ is a splitting field for $iB^Pi$, then  the following hold.
\begin{itemize}
\item[{\rm (i)}]
The $B$-$kS$-bimodule $Bi\tenkP kS$ is isomorphic to a direct summand of $B$. 
\item[{\rm (ii)}] 
The $B$-module $Bi\tenkP k$ is isomorphic to  a direct summand of $B\tenkS k$.
\item[{\rm (iii)}]
The $B$-module  $Bi\tenkP k$ has an indecomposable direct  summand with vertex $P$.
\item[{\rm (iv)}]
Every indecomposable direct summand of $B\tenkS k$ with vertex $P$ is isomorphic 
to a direct summand of $Bi\tenkP k$. 
\end{itemize}
\end{Proposition}

\begin{proof}
The statements are independent of the choice of $i$ in $\gamma$; in particular,
we may replace $i$ by any $(B^P)^\times$-conjugate. Thus we may choose $i$
such that there is a primitive idempotent $f$ in $B^S$ satisfying
$if=i=fi$. Let $\nu$ be the point of $S$ on $B$ containing $f$. By construction
we have $P_\gamma\leq S_\nu$. Since $P_\gamma$
contains a $G$-conjugate of every local pointed group on $B$ it follows that
$P_\gamma$ is a defect pointed group of $S_\nu$.
The statements (i), (ii) and (iii) follow from Proposition \ref{kGi-summands-Prop},
applied with $S_\nu$, $P_\gamma$ instead of $P_\gamma$, $Q_\delta$.
For (iv), let $U$ be an indecomposable direct summand of $B\tenkS k$ with
vertex $P$. Let $f$ be a primitive idempotent in $B^S$ such that $U$ is
isomorphic to a direct summand of $Bf\tenkS k$. Denote by $\mu$ the point
of $S$ on $B$ containing $f$, and let $Q_\delta$ be a defect local pointed 
group of $S_\mu$. It follows from Lemma \ref{kGi-summands-R} (i) that
$U$ is isomorphic to a direct  summand of $kGj\tenkQ k$, for some $j\in$ $\delta$.
Since $P$ is a vertex of $U$ it follows that $Q$ is conjugate to $P$. But then
$Q_\delta$ is conjugate to $P_\gamma$, and hence (iv) follows from
Lemma \ref{isomorphic-summands-Lemma}. 
\end{proof}

\begin{proof}[Proof   of Theorem~\ref{Bi-summand}]  
This is  part   (ii)  of  Proposition \ref{BiBS-summand-Prop}.
\end{proof} 
For a suitable choice of $j$ in  $\delta$, the isomorphism in Proposition
\ref{kGi-summands-Prop} (ii) is  given by multiplication in $kG$.

\begin{Lemma} \label{kGi-summand-isomorphism}
Let $G$ be a finite group, and let $P_\gamma$, $Q_\delta$ be pointed groups
on $kG$ such that $Q_\delta\leq P_\gamma$ and such that $Q_\delta$ is
a defect pointed group of $P_\gamma$. Let $i\in \gamma$. 
Suppose that $k$ is a splitting field for $(ikGi)^R$, where $R$ runs over the
subgroups of $P$.
Then there exists $j\in$ $\delta$ such that the map $kGj\tenkQ kP\to$ $kGi$
sending $aj\ten u$ to $aju$ is an isomorphism, where $a\in kG$ and $u\in P$.
\end{Lemma}

\begin{proof}
As a consequence of  Puig's generalisation of Green's Indecomposability Theorem  
(see  e.g. \cite[Corollary 5.12.21]{LiBookII}), there exists $j\in \delta$ such that
the different $P$-conjugates of $j$ are pairwise orthogonal and such that
$i=$ $\Tr^P_Q(j)$. Thus 
$$kGi= kG\Tr^P_Q(j)= \oplus_{u}\ kGu^{-1}ju= \oplus_{u}\ kGju,$$
where $u$ runs over a set of representatives in $P$ of $Q\backslash P$. 
The result follows.
\end{proof}

We note that we need in this Lemma a slightly stronger hypothesis on $k$
 being large enough compared to the hypothesis in Proposition 
\ref{kGi-summands-Prop}. This  is due to the inductive nature of the proof of
Puig's generalisation of Green's Indecomposability Theorem in 
\cite[Theorem 5.12.20]{LiBookII}.

From Lemma \ref{kGi-summand-isomorphism} we can be slightly more
precise regarding statement (ii) in Proposition \ref{BiBS-summand-Prop}.

\begin{Corollary} \label{BjB-split-Cor}
Let $G$ be a finite group, $S$ a Sylow $p$-subgroup of $G$ and $B$ a block
of $kG$.  Assume that $k$ is a splitting field for $B^Q$,where $Q$ runs over the
$p$-subgroups of $G$. Let $P_\gamma$ be a defect local pointed group of $B$. 
Suppose that $P \leq S$.
There exists $i\in \gamma$ such that the inclusion map $Bi\to B$ induces
a split injective $B$-homomorphism $Bi\tenkP k\to$ $B\tenkS k$.
\end{Corollary}

\begin{proof}
By Proposition \ref{BiBS-summand-Prop}, there is a point $\nu$ of $S$ on
$B$ such that $P_\gamma$ is a defect pointed group of $S_\nu$.
Thus the result  is a special case of Lemma \ref{kGi-summand-isomorphism}.
\end{proof}
We do not know under what conditions the canonical surjection $B\tenkP k\to$ 
$B\tenkS k$ splits.  
Our next task is  to  prove   Theorem  \ref{Bi-no-projective-summand-thm}.

\begin{Lemma} \label{no-proj-summands-Lemma}
Let $G$ be a finite group, and let $P_\gamma$, $Q_\delta$  be pointed groups
on $kG$ such that $Q_\delta\leq P_\gamma$ and such that $Q_\delta$ is
a defect pointed group of $P_\gamma$. Let $i\in \gamma$.
\begin{itemize}

\item[{\rm (i)}]
If $Q\neq 1$, then the $kG$-module $kGi \tenkP k$ has no nonzero projective 
direct summand, and for any  simple $kG$-module $Y$ the $kP$-module $iY$ has
no nonzero projective direct summand.
\item[{\rm (ii)}]
If $Q=1$, then  $kGi$ is a projective indecomposable $kG$-$kP$-bimodule, and
$kGi\tenkP k$ is a projective indecomposable $kG$-module.
\end{itemize}
\end{Lemma}

\begin{proof}
Suppose that $Q\neq 1$. Then  by Lemma \ref{bimod-idempotent-Lemma} the $kG$-$kP$-bimodule $kGi$ is 
indecomposable non-projective, and its restriction to $kP$ on the right
is projective. Similarly,  the $kP$-$kG$-bimodule $ikG$ is indecomposable 
non-projective, and its restriction as a right $kG$-module is  projective.   Thus 
both statements in  (i) follow from  Proposition \ref{no-proj-summands-Prop}.
If $Q=1$, then $kGi$ is a projective indecomposable  $kG$-$kP$-bimodule,
hence isomorphic to $kGj\tenk kP$ for some primitive idempotent $j$
(we use here Lemma \ref{ij-primitive-Lemma} and the fact that $kP$ is split local).
Tensoring with $-\tenkP k$ implies (ii). 
\end{proof}

We spell out some special cases of this Lemma that are frequently useful.

\begin{Lemma} \label{proj-summands-Lemma}
Let $G$ be a finite group, $B$ a block of $kG$ with a non-trivial  defect group $P$, 
and let $i\in B^P$ a source idempotent. Set $A=iBi$. 

\begin{itemize}
\item[{\rm (i)}] 
For any simple left or right $A$-module $X$, the restriction $\Res_P(X)$ of $X$
to $kP$ has no nonzero projective direct summand.

\item[{\rm (ii)}]
The $A$-module $A\tenkP k$ has no nonzero projective direct summand.

\item[{\rm (iii)}]
The $B$-module $Bi\tenkP k$ has no nonzero projective direct summand.
\end{itemize} 
\end{Lemma}

\begin{proof}
Multiplication by $i$ induces a Morita equivalence between $B$ and its source
algebra $A$, and hence every simple $A$-module $X$ is isomorphic to $iY$
for some simple $B$-module $Y$. Thus statement (i) is a special case of Lemma 
\ref{no-proj-summands-Lemma} (and is,  for instance, proved in 
\cite[Proposition 6.4.11]{LiBookII}).
The statements (ii) and (iii) follow from Proposition \ref{no-proj-summands-Prop}.
Alternatively, statement (iii) follows from Lemma \ref{no-proj-summands-Lemma}
(i), and statement (ii) follows from (iii) and the Morita equivalence
between $B$ and $A$ as before.
\end{proof}

\medskip
A standard argument using Frobenius reciprocity shows that every simple
$kG$-module is isomorphic to a quotient and to a submodule of the Sylow
permutation $kG$-module, and hence, for any block $B$,
any simple $B$-module is isomorphic to a quotient and to a submodule of
$B\tenkS k$. The following observation extends this to source permutation
modules. 

\begin{Proposition} \label{Bi-top-socle}
Let $G$ be a finite group, $B$ a block of $kG$, $P$ a defect group of $B$,
and $i$ an idempotent in $B^P$ such that $\Br_P(i)\neq 0$.
Then every simple $B$-module is isomorphic to a quotient and to a
submodule of $Bi\tenkP k$.
\end{Proposition}

\begin{proof}
We first note that $Bi$, regarded as a $B$-$kP$-bimodule is dual to the
$kP$-$B$-module $iB$ because $B$ is symmetric. Thus
the functors $Bi\tenkP -$ and $iB \ten_B -$ between $\mod(kP)$ and $\mod(B)$
are biadjoint (see e.g. \cite[Theorem 2.12.7]{LiBookI}). 
The hypothesis $\Br_P(i)\neq 0$ implies that the algebras $B$ and $iBi$ are 
Morita equivalent. In particular, given a simple $B$-module $X$, we have $iX\neq 0$.  
Thus there are nonzero $kP$-homomorphisms
$iX\to k$ and $k\to iX$. Via the above adjunction, these correspond to
nonzero $B$-homomorphisms $X\to Bi\tenkP k$ and $Bi\tenkP k\to$ $X$,
whence the result.
\end{proof}

\begin{proof} [Proof of Theorem \ref{Bi-no-projective-summand-thm}]  
Part (i)  is  part (iii) of  Lemma~\ref{proj-summands-Lemma}  and Part (ii) 
is  given by  Proposition~\ref{Bi-top-socle}.
\end{proof}

For future reference, we make  some elementary observations
about blocks with a normal defect group; see \cite{Kuels85} for more details on
the structure of  such blocks.

\begin{Proposition} \label{normal-Prop}
Let $G$ be a finite group and $B$ a block of $kG$ with a normal defect group 
$P$.  Let $i\in B^P$ be an idempotent satisfying $\Br_P(i)\neq 0$.  
\begin{itemize}
\item[{\rm (i)}]
The $B$-modules $B\tenkP k$ and $Bi\tenkP k$ are semisimple.
\item[{\rm (ii)}]
 Every simple $B$-module is isomorphic to a direct summand of $B\tenkP k$
 and of $Bi\tenkP k$.
 \end{itemize}
 \end{Proposition}
 
\begin{proof} 
Set $b=$ $1_B$. Then $b=\Tr^G_P(d)$ for some $d\in B^P$. Denote by $\bar b$, 
$\bar d$ the images of $b$ and $d$ in $kG/P$, respectively. Then $\bar b=$ 
$\Tr^{G/P}_1(\bar d)$, so $\bar B=$ $kG/P\bar b$ is a direct product of defect zero 
blocks of $kG/P$.  Thus every finitely generated $\bar B$-module is semisimple, or
equivalently, every finitely generated $B$-module on which $P$ acts 
trivially is semisimple. Using that $P$ is normal in $G$, clearly $P$ acts trivially 
on $B\tenkP k$, hence on  $Bi\tenkP k$. This shows (i). 
Statement (ii) is a special case of Proposition \ref{Bi-top-socle}.
\end{proof}

We   end  this section by recording a  relationship between    the 
source permutation module of a block  and  that of its Brauer correspondent.

\begin{Proposition} \label{source-perm-module-Brauer-corr}
Let $G$ be a finite group, $B$ a block of $kG$ with a non-trivial defect group $P$,
and let $C$ be the block of $kN_G(P)$ which is the Brauer correspondent of $B$.
Let $j\in$ $C^P$ and $i\in B^P$ be source idempotents. 
As a $kG$-module,  $Bi\tenkP k$ is isomorphic to a direct summand of 
$\Ind^G_{N_G(P)}(Cj\tenkP k)$.
\end{Proposition}

\begin{proof}
By   Lemma \ref{isomorphic-summands-Lemma} the isomorphism class
of $Bi\tenkP k$ is independent of the choice of $i$. 
Since $\Br_P(j)\neq 0$, we may choose $i$ such that $ij=i=ji$.
Thus, as a $kG$-module, $Bi\tenkP k$ is isomorphic to a direct summand of
$kGj\tenkP k\cong$ $\Ind^G_{N_G(P)}(Cj \tenkP k)$.
\end{proof}

\section{Weight modules, and proofs of Theorems \ref{Bi-weight-thm} and
\ref{Bi-selfinjective-thm1} } \label{weight-section}

Let $k$ be a field of prime characteristic $p$. Let $G$ be a finite group and
$B$ a block of $kG$.
The  main  Theorem of this section  shows that $B$-modules corresponding to weights are 
summands of the source permutation module.  A {\em $G$-weight} is a pair $(Q,X)$
consisting of a $p$-subgroup $Q$ of $G$ and a projective simple
$kN_G(Q)/Q$-module.  A {\em weight subgroup of $G$} is a $p$-subgroup $Q$ of
$G$ such that there exists a projective simple $kN_G(Q)/Q$-module $X$, or
equivalently, such that $(Q,X)$ is a $G$-weight. In that case, inflating $X$ to 
$kN_G(Q)$ yields a simple $kN_G(Q)$-module with vertex $Q$ and trivial source. 
The Green correspondent of $X$ is an indecomposable $kG$-module with vertex 
$Q$  and trivial source. The group $G$ acts by conjugation on isomorphism
classes of weights. 
 
Alperin's weight conjecture \cite{Alp87}  in the group 
theoretic version predicts that if the field $k$ is a splitting field for normalisers
of $p$-subgroups, then the number of conjugacy classes of $G$-weights 
should be equal  to the number of isomorphism classes of simple $kG$-modules. 

Given  a block $B=$ $kGb$ of $kG$, a {\em $B$-weight} is a $G$-weight $(Q,X)$ such 
that the Green correspondent $U$ of $X$ as a $kN_G(Q)$-module belongs to $B$.
In that case $X$ belongs to a block $C=kN_G(Q)c$  of $kN_G(Q)$ such that 
$\Br_Q(b)c=$ $c$. Since every block idempotent of $kN_G(Q)$ is contained in
$kC_G(Q)$ there is a block $kC_G(Q)e$ such that $c=\Tr^{N_G(Q)}_{N_G(Q,e)}(e)$.
Then $(Q,e)$ is a {\em centric}  $B$-Brauer pair (that is, $Z(Q)$ is a defect group of 
$kC_G(Q)e$), and $(Q,e)$ is {\em radical} (that is, $O_p(N_G(Q,e))=Q$). 
(If $k$ is a splitting field for the Brauer pairs of $B$,
then, more precisely, $Q$ is a radical centric subgroup  in a fusion system of $B$ 
on a  defect group containing $Q$; see e.g. \cite[Theorem 6.10.8]{LiBookII}.)
The blocks $kN_G(Q)c$ and $kN_G(Q,e)e$ are Morita equivalent via induction
and truncated restriction, and hence projective simple $k(N_G(Q)/Q)  \bar c$-modules
correspond to projective simple $k(N_G(Q,e)/Q)\bar e$-modules, where $\bar c$ and 
$\bar e$ are the obvious images of $c$ and $e$. Thus conjugacy classes of $B$-weights
correspond bijectively to conjugacy classes of triples $(Q,e,Y)$, where
$(Q,e)$ is a centric radical  $B$-Brauer pair and $Y$ a projective simple 
$kN_G(Q,e)/Q\bar e$-module. If $B$ has a normal defect group $P$, then
the only weight subgroup is $P$ itself, and the $B$-weights are of the
form $(P,X)$, where $X$ is a simple $B$-module; that is, we have a
canonical bijection between isomorphism classes  of simple $B$-modules and
of conjugacy classes of $B$-weights in that case.

The block theoretic version of Alperin's weight conjecture in \cite{Alp87}
predicts that for any block $B$ of $kG$,  if $k$ is large enough for $B$, its Brauer 
pairs and  their normalisers,  then
the number of isomorphism classes $\ell(B)$ of simple $B$-modules should be 
equal to the number of conjugacy classes $w(B)$ of $B$-weights. The block theoretic
version of Alperin's weight conjecture implies the group theoretic version,
by taking sums over all blocks. The following Theorem restates 
Theorem \ref{Bi-weight-thm} in a slightly more precise way.

\begin{Theorem} \label{Bi-weights}
Let $G$ be a finite group, $B$ a block of $kG$, $P$ a defect group of $B$, and
$i\in B^P$ a source idempotent. Let $Q$ be a subgroup of $P$ and
$X$ a projective simple $kN_G(Q)/Q$-module such that the Green correspondent
$U$ of $X$ regarded as a $kN_G(Q)$-module belongs to $B$. Then $U$ is
isomorphic to a direct summand of the source permutation module $Bi\tenkP k$.
In particular, $Bi\tenkP k$ has at least $w(B)$ isomorphism classes of
indecomposable direct summands.
\end{Theorem}

The proof will make use of the following lemma.

\begin{Lemma} \label{Bi-summands-Lemma}
Let $G$ be a finite group, $B$ a block of $kG$, $P$ a defect group of $B$, and
$i\in B^P$ a source idempotent. Let $W$ be an indecomposable trivial source
$B$-module. Then $W$ is isomorphic to a direct summand of $Bi \tenkP kP/R$ for 
some vertex $R$ of $W$ such that $R\leq P$. 
\end{Lemma}

\begin{proof}
Denote by $\gamma$ the local point of $P$ on $B$ which contains $i$.   Let $R$ 
be a vertex of $W$. Then $W$ is isomorphic to a direct summand of $B\tenkR k$. 
Since $W$ is indecomposable, $W$ is isomorphic to a direct summand of 
$Bj\tenkR k$  for some primitive idempotent $j$ in $B^R$. Since $R$ is a vertex 
of $W$, it  follows that $j$ belongs to a local point $\epsilon$ of $R$ on $B$. 
Some $G$-conjugate of $R_\epsilon$ is contained in $P_\gamma$. Thus, by 
Lemma \ref{isomorphic-summands-Lemma}, we may assume that 
$R_\epsilon\leq$ $P_\gamma$, and we may choose $j\in$ $\epsilon$ such that 
$ij=j=ji$ and such that $W$ is isomorphic to a direct summand of $Bj\tenkR k$. 
But then $W$ is isomorphic to a direct summand of $Bi\tenkR k\cong$
$Bi \tenkP kP\tenkR k\cong$ $Bi\tenkP kP/R$. The Lemma follows.
\end{proof}

\begin{proof}[Proof of Theorem \ref{Bi-weights}]
Recall  that  the functors $Bi\tenkP -$ and $iB \ten_B -$ between 
$\mod(kP)$ and $\mod(B)$ are biadjoint, and note further that
$iB\ten_B U\cong$ $iU$.   Since $U$ is the Green correspondent of the
inflation to $N_G(Q)$ of a projective simple $kN_G(Q)/Q$-module, it follows
in particular that $U$ has vertex $Q$ and a trivial source. 
We need to show that there is a split injective $B$-homomorphism
$ \alpha :U \to Bi\tenkP k.$
We start by constructing a
$kP$-homomorphism  $\beta : iU\to k$ as follows.
By Corollary  \ref{iU-permutation-summands}, up to replacing $Q$ by a conjugate,
we may assume that $iU$ has a direct summand isomorphic to $kP/Q$ as
a $kP$-module. Define a $kP$-homomorphism $\beta : iU \to k$ as the
composition of a split surjective $kP$-homomorphism $iU\to kP/Q$
and a surjective $kP$-homomorphism $kP/Q\to k$. Applying the Brauer
construction to $\beta$ yields thus a nonzero $kN_P(Q)/Q$-homomorphism
$\beta(Q) : (iU)(Q)\to k$.

Denote by $\alpha : U \to Bi \tenkP k$ the $B$-homomorphism corresponding
to $\beta$ through an adjunction
$$\Hom_B(U, Bi\tenkP k)\cong \Hom_{kP}(iU, k).$$ 
 Note that all modules involved in this isomorphism are $p$-permutation modules.
Thus, by \cite[Proposition A.1]{Li-endo} (or \cite[Proposition 5.10.7]{LiBookI})
it suffices to show that the map
$$\alpha(Q) : U(Q) \to (Bi\tenkP k)(Q)$$
is injective. Now thanks to a result of Brou\'e \cite[(3.4)]{Broue85} 
(see e.g. \cite[Theorem 5.10.5]{LiBookI}), we have $U(Q)\cong$ $X$.
Since $X$ is simple, it suffices to show that $\alpha(Q)$ is nonzero. 

Arguing by contradiction, assume that $\alpha(Q)=0$. Since
$U$ and $Bi\tenkP k$ are $p$-permutation modules, it follows from
Lemma \ref{Brauer-p-permutation}
that $\Hom_k(U(Q),(Bi\tenkP k)(Q))$ can be identified with the Brauer
construction  $(\Hom_k(U, Bi\tenkP k))(Q)$ with respect to the diagonal
action of $Q$, and hence $\alpha(Q)$ can be identified with the image
of $\alpha$ in $(\Hom_k(U, Bi\tenkP k))(Q)$. This image of $\alpha$ is zero 
if and only if $\alpha$ is contained in 
$$\Hom_{kG}(U,Bi\tenkP k) \cap \ker(\Br_Q)$$
where here $\ker(\Br_Q)$ is the kernel of the Brauer construction applied to 
$\Hom_k(U, Bi\tenkP k)$.  By Lemma \ref{Hom-p-permutation} this space is
equal to
$$\sum_{R<Q}\ \Hom_k(U,Bi\tenkP k)_R^G,$$
where $R$ runs over the proper subgroups of $Q$. Thus, since we assume
$\alpha(Q)=0$, we may write $\alpha$ in this space as a sum
$$\alpha \ =\ \sum_{R<Q}\ \Tr^G_R(\alpha_R),$$
with $\alpha_R\in$ $\Hom_{kR}(U, Bi\tenkP k)$ for each proper subgroup $R$ of $Q$.
Each such $\Tr^G_R(\alpha_R)$ is a $kG$-homomorphism from $U$ to $Bi\tenkP k$,
and since this is a trace from $R$, it follows from \cite[Corollary 2.6.8]{LiBookI}
(which is a consequence of Higman's criterion) that $\Tr^G_R(\alpha_R)$
factors through the canonical surjection $B\tenkR Bi\tenkP k\to Bi\tenkP k$
sending $a\ten ci \ten 1$ to  $aci\ten 1$, where $a$, $c\in B$. 
The module $B\tenkR Bi\tenkP k$ is
a $p$-permutation $B$-module all of whose indecomposable summands have
vertices contained in $R$.  By  Lemma \ref{Bi-summands-Lemma} the 
indecomposable summands of this module are summands of modules of the 
form $Bi \tenkP kP/R'$ for some subgroup $R'$ conjugate to a subgroup of $R$. 
It follows that $\alpha$ factors through a direct sum of $B$-modules of the form
$$\oplus_{R}\ Bi \tenkP kP/R$$
with $R$ running over a family of subgroups conjugate to proper subgroups
of $Q$. That  is, we have a commutative diagram of the form
$$\xymatrix{ & & Bi \tenkP k \\
U \ar[rru]^{\alpha} \ar[rrd]_\tau & & \\
& & \oplus_{R}\ Bi \tenkP kP/R \ar[uu]_{\sigma} 
}$$
where as before $R$ runs over a family of subgroups $R$ which are conjugate to 
proper subgroups of $Q$. Making use of the functoriality of the adjunction 
isomorphism
$$\Hom_B(U, Bi\tenkP -)\cong \Hom_{kP}(iU, -)$$
we get a commutative diagram of $kP$-modules 
$$\xymatrix{ & & k \\
iU \ar[rru]^{\beta} \ar[rrd] & & \\
& & \oplus_{R}\ kP/R \ar[uu] 
}$$
Since the groups $R$ in the sum have all order strictly smaller than $Q$
it follows that $\beta(Q)=0$, a contradiction. This shows that $\alpha(Q)$
is nonzero, hence injective, and hence $\alpha$ is split injective by the above.
The last statement follows immediately. 
\end{proof}

\begin{Corollary} \label{ellC-Cor}
Let $G$ be a finite group  with a Sylow subgroup $S$, $B$ a block of $kG$ with 
a non-trivial defect group $P$, and let $C$ be the block of $kN_G(P)$ which is 
the Brauer correspondent of $B$. Let $ i\in B^P$ be a  source idempotent. 
Suppose that $k$ is a splitting field for $iB^Pi$.
The number of isomorphism classes of indecomposable direct summands with
vertex $P$ of $Bi\tenkP k$  and  of  $B\tenkS k$  is equal to $\ell(C)$. 
\end{Corollary}

\begin{proof}
By Proposition \ref{BiBS-summand-Prop}  it suffices to prove the  equality for   
$Bi\tenkP k$.  Let $ j \in C^P$ be a source idempotent.
Since $\Br_P(j)\neq 0$, we may,  after suitable replacement,   assume that  $ij=i=ji$.
Thus, as a $kG$-module, $Bi\tenkP k$ is isomorphic to a direct summand of
$kGj\tenkP k\cong$ $\Ind^G_{N_G(P)}(Cj \tenkP k)$. 
Now  by Proposition \ref{normal-Prop}, $Cj\tenkP k$
is  semisimple  and has $\ell(C)$ isomorphism classes of simple
summands, and these all have vertex $P$ as $P$ is normal in $N_G(P)$.
The Green correspondence implies that $\Ind^G_{N_G(P)}(Cj\tenkP k)$
has exactly $\ell(C)$ isomorphism classes of indecomposable direct summands
with vertex $P$. Thus $Bi\tenkP k$ has at most $\ell(C)$ isomorphism 
classes of indecomposable direct summands with vertex $P$. 
Since the $\ell(C)$ isomorphism classes of simple $C$-modules give rise
to $B$-weights, it follows from Theorem \ref{Bi-weights} that $Bi\tenkP k$ 
has exactly $\ell(C)$ isomorphism classes of indecomposable direct
summands with $P$ as a vertex.
\end{proof}

Combining the above results yields  the following  version of an observation in 
Naehrig \cite[Remark 2(b)]{Nae}, with the Sylow permutation module replaced 
by the source permutation module. 

\begin{Corollary} \label{Bi-selfinjective-Cor}
Let $G$ be a finite group,  $B$ a block of $kG$ 
with a non-trivial defect group $P$, and let $i\in B^P$  be a source idempotent.
Suppose that $k$ is a splitting field for all $B$-Brauer pairs and their normalisers.
Assume that $\End_B(Bi\tenkP k)$ is self-injective.  Then $\ell(B)\geq w(B)$. 
Moreover, the equality  $\ell(B)=w(B)$, as predicted by Alperin's weight conjecture,
holds if and only if every indecomposable  direct summand of $Bi\tenkP k$ is the 
Green correspondent of a $B$-weight.
\end{Corollary}

\begin{proof}
By  Theorem \ref{Bi-no-projective-summand-thm} every simple
$B$-module is isomorphic to a quotient and a submodule of $Bi\tenkP k$,
and hence $Bi\tenkP k$ has $\ell(B)$ isomorphism classes of indecomposable
direct summands by a theorem of Green (cf. \cite[Theorem 2.13]{Geck01}).
But $Bi\tenkP k$ has  at least $w(B)$ pairwise non-isomorphic summands
arising as Green correspondents of $B$-weights, by Theorem \ref{Bi-weights}, 
whence the result.
\end{proof}

\begin{proof}[{Proof of Theorem \ref{Bi-selfinjective-thm1}}]
With the notation and hypotheses of Theorem \ref{Bi-selfinjective-thm1}, 
suppose that the algebra $\End_B(B\tenkS k)$ is self-injective. 
As remarked before Proposition \ref{Bi-top-socle}, every simple $B$-module
is isomorphic to a quotient and a submodule of $B\tenkS k$. Thus
Green's theorem  \cite[Theorem 1]{Green78} (see also \cite[Theorem 2.13]{Geck01}) 
implies that every indecomposable
direct summand of $B\tenkS k$ has a simple top and socle, and that
the isomorphism class of any such summand is determined by the isomorphism
class of its top, as well as by that of its socle. That is, we have
bijections with between the isomorphism classes of indecomposable direct
summands of $B\tenkS k$ and simple $B$-modules obtained by sending such
a summand to its top (resp. its socle).
Since   by Theorem \ref{Bi-summand},  $Bi\tenkP k$ is isomorphic to a direct  
summand of $B\tenkS k$ 
 and since every simple $B$-module appears both 
in the top and the socle of $Bi\tenkP k$, up to isomorphism 
(cf. Theorem \ref{Bi-no-projective-summand-thm}),  it follows that
every indecomposable direct summand of $B\tenkS k$ is isomorphic to
a direct summand of $Bi\tenkP k$. Thus these two modules have Morita
equivalent endomorphism algebras as stated. 
\end{proof}

\begin{proof}[{Proof of Corollary \ref{noprojective-cor}}]
Note that $Bi\tenkP k$ has no nonzero projective summand (cf. 
Lemma \ref{proj-summands-Lemma}).
Thus if $B\tenkS k$ has a nonzero projective summand, then 
the endomorphism algebras of $Bi\tenkP k$ and of $B\tenkS k$ are
not Morita equivalent and Theorem \ref{Bi-selfinjective-thm1} implies that 
$\End_B(B\tenkS k)$ 
is not self-injective. The second  statement follows from the first since  
any  non-simple projective indecomposable  $kG$-module    
belongs to a block with  non-trivial defect groups.
\end{proof}

We end  this section with an observation  on splitting.

\begin{Proposition} \label{splitting-Lemma}
Let  $B$  be  a block of $kG$. Suppose that $k$ is a 
splitting  field for $B$ and all $B$-Brauer pairs.
Then, for any $p$-subgroup $P$ of $G$ the field $k$ is a splitting 
field for $B^P$.   Equivalently, for any $p$-subgroup $P$ of $G$ and 
any primitive idempotent $i$ in $B^P$ the $kG$-$kP$-bimodule $Bi$ is 
absolutely indecomposable.
\end{Proposition}

\begin{proof}
The two statements are indeed equivalent by Proposition 
\ref{kGi-summands-Prop} (ii). We will show the second statement. Let $P$ be a 
$p$-subgroup and $i$ a  primitive idempotent in $B^P$. Denote by $\gamma$ 
the point of $P$ on  $B$  containing $i$, and let  $Q_\delta$ be a defect 
pointed  group of  $P_\gamma$. Consider the corresponding $B$-Brauer pair 
$(Q,e)$, where  $e$ is the block idempotent of $kC_G(Q)$ such that $
\Br_Q(\delta)e\neq 0$. Let $j\in$ $\delta$. Since $j$ is primitive in $B^Q$, it 
follows that $\Br_Q(j)$ is primitive in $kC_G(Q)e$. By the assumptions, $k$ is 
a splitting field for $kC_G(Q)e$. Thus $\Br_R(j)kC_G(Q)\Br_R(j)=$ 
$\Br_R((jkGj)^Q)$ is split local, and hence $(jkGj)^Q$ is split local. It follows 
again from Proposition \ref{kGi-summands-Prop} (ii) that $(ikGi)^P$ is split 
local, or equivalently, the $kG$-$kP$-bimodule $kGi$ is absolutely 
indecomposable.
\end{proof}

\section{Source permutation modules for principal blocks, and proof of
Theorem \ref{princ-block-thm} }
\label{princ-block-Section}

For principal blocks we have more precise results which provide an alternative
proof of Theorem \ref{Bi-weights}  in this case.  Let $k$ be a field
of prime characteristic $p$. A  $p$-subgroup $Q$ of   a finite group $G$ is called 
{\em centric in} $G$ if $Z(Q)$ is a Sylow $p$-subgroup of $C_G(Q)$.  
A $p$-subgroup $Q$ of $G$ is called {\em essential} if $Q$ is centric in $G$ 
and  $N_G(Q)/Q$ has a strongly $p$-embedded  subgroup. 

\begin{Theorem} \label{Bi-princ-block}
Let $G$ be a finite group, $B$ the principal block of $kG$, $S$ a Sylow $p$-subgroup
of $G$, and $i\in B^S$ a source idempotent. Let $U$ be an indecomposable
direct summand of $B\tenkS k$. If $U$ has a vertex  which is centric in $G$,
then $U$ is isomorphic to a direct summand of the source permutation module
$Bi\tenkS k$.
\end{Theorem}

The proof of Theorem \ref{Bi-princ-block} 
is based on the following Lemmas, for which we briefly review
some terminology and notation.  As before, a  Brauer pair $(Q,e)$ on $kG$
is called {\em centric} if $Z(Q)$ is a defect group of the block $kC_G(Q)e$.
If $e$ is the principal block of $kC_G(Q)$, then $(Q,e)$ is a centric
Brauer pair if and only if   $Z(Q)$ is a Sylow $p$-subgroup of $C_G(Q)$, so if and
only if $Q$ is centric in $G$. For a general Brauer pair $(Q,e)$ on $kG$, assume
that $k$ is a splitting field for $kC_G(Q)e$ and that $(Q,e)$ is centric. 
Then $kC_G(Q)e$ is a nilpotent block, so  has a unique isomorphism class of simple 
modules, and hence there is a unique local point $\delta$ of $Q$ on $kG$ such that
$\Br_Q(\delta)e\neq 0$. Moreover, $kC_G(Q)/Z(Q)\bar e$ is isomorphic to a 
matrix algebra over $k$, where
$\bar e$ is the image of $e$ in $kC_G(Q)/Z(Q)$, and we may write
$$k C_G(Q)/Z(Q)\bar e\cong\End_k(V_Q)$$ 
for some vector space $V_Q$.
If $R \leq G $ is a $p$-group containing $Q$ and  normalising $(Q, e) $, then 
$R$ acts on $kC_G(Q)e$, hence induces an action on $\End_k(V_Q)$ 
which lifts canonically to an action of $R$ on $V_Q$  with $Q$  acting trivially. 
In this way,  $V_Q$  inherits  via restriction  the structure of   an endopermutation 
$kR/Q$-module.
The first Lemma is a special case of the parametrisation of modules
with a given vertex-source pair in terms of their multiplicity modules;
we sketch a proof since the particular case is simpler than the general situation
(for which we refer to \cite[Section 5.7]{LiBookI} for more details and references).

\begin{Lemma} \label{non-local-points-1}
Let $G$ be a finite group, $B$ a block of $kG$ and $(Q,e)$ a centric $B$-Brauer 
pair. Assume that $k$ is a splitting field for $kC_G(Q)e$. Denote by $Q_\delta$ 
the unique local point of $Q$ on $B$ such that $\Br_Q(\delta)e\neq 0$.
Let $P$ be a $p$-subgroup of $G$ containing $Q$ such that $N_P(Q)$
stabilises $(Q,e)$. There is a point $\nu$ of 
$P$ on $B$ such that $Q_\delta$ is a defect pointed group of $P_\nu$ 
if and only  if $V_Q$ has a nonzero projective summand as a $kN_P(Q)/Q$-module,
where $V_Q$ is a $kN_P(Q)/Q$-module satisfying $kC_G(Q)/Z(Q)\bar e\cong$ 
$\End_k(V_Q)$ as $N_P(Q)$-algebras, and where $\bar e$ is the image of $e$
in $kC_G(Q)/Z(Q)$.  Moreover, if such a point $\nu$ exists, then it is unique. 
\end{Lemma}

\begin{proof}
For any $d\in$ $(kG)^Q$  we have 
$$\Br_Q(\Tr^P_Q(d))= \Tr^{N_P(Q)}_Q(\Br_Q(d));$$  
see e.g. \cite[Proposition 5.4.5]{LiBookI} for this  well-known formula. 
Thus $\Br_Q$ maps $(kG)^P_Q$ onto  
$$kC_G(Q)^{N_P(Q)}_Q =  kC_G(Q)^{N}_1,$$ 
where $N=N_P(Q)/Q$. The canonical surjective map $kC_G(Q)\to$ $kC_G(Q)/Z(Q)$ 
sends  $(kC_G(Q))_1^{N}$ onto $(kC_G(Q)/Z(Q))^{N}_1$. Thus if $\nu$ is a point of 
$P$ with defect group $Q$, then  $\Br_Q(\nu)$ is a point in $(kC_G(Q)e)_1^{N}$, 
where we use that $N_P(Q)$, hence $N$,  stabilises $e$. Therefore,  the image of 
$\Br_Q(\nu)$ in 
$$(kC_G(Q)/Z(Q)\bar e)_1^{N}\cong \End_k(V_Q)_1^{N}$$ 
is a point which, by Higman's criterion,  corresponds to the projective 
indecomposable summands of $V_Q$ as a  $kN$-module.
Conversely, if $V_Q$ has a nonzero projective indecomposable summand as a 
$kN$-module, then by standard lifting theorems, reversing the above reasoning, 
there is a unique point $\nu$ of  $P$ on $B$ with $Q_\delta$ as defect pointed group.
\end{proof}

\begin{Lemma} \label{non-local-points-2}
Let $G$ be a finite group, $S$ a Sylow $p$-subgroup of $G$, and denote by
$B$ the principal block of $kG$. Let $\nu$ be a non-local point of $S$ on $B$.
Let $Q_\delta$ be a defect pointed group of $S_\nu$. Then $Q$ is not centric
in $G$.
\end{Lemma}

\begin{proof}
Arguing by contradiction, assume that $Q$ is centric in $G$. Note that
$Q$ is a proper subgroup of $S$ since $\nu$ is not local, and hence
$N_S(Q)/Q$ is a non-trivial $p$-group. Moreover, the block $e=\Br_Q(1_B)$ 
is the principal block of $kC_G(Q)$ (where we use Brauer's Third Main Theorem
\cite[Theorem 6.3.14]{LiBookII}). Being centric implies that  $C_G(Q)\cong$ 
$C'\times Z(Q)$ for some $p'$-subgroup $C'$ of $C_G(Q)$. The trivial $kC'$-module 
determines the principal block idempotent $e$, and we have $kC_G(Q)e\cong$ 
$kZ(Q)$, hence $kC_G(Q)/Z(Q)\bar e\cong$ $k$. In particular, $k$ is a splitting 
field for $kC_G(Q)e$. Thus the vector space $V_Q$ from Lemma 
\ref{non-local-points-1} has dimension $1$. In particular, $V_Q$ has no nonzero 
projective summand as a $kN_S(Q)/Q$-module. Thus, by Lemma 
\ref{non-local-points-1}, there is no point $\nu$ of $S$ on $B$ with defect 
pointed group $Q_\delta$.
\end{proof}

\begin{proof}[{Proof of Theorem \ref{Bi-princ-block}}]
Since $B$ is the principal block of $kG$, we have $B(S)\cong$ $kZ(S)$.
Thus $\Br_S(1_B)=\Br_S(i)=1_{kZ(S)}$, and hence, if $\nu$ is any point of
$S$ on $B$ different from the local point $\gamma$ containing $i$, then
$\nu$ has a defect pointed group $Q_\delta$ with $Q$ a proper subgroup
of $S$. But then $Q$ is not centric, by Lemma \ref{non-local-points-2},
and hence no subgroup which is $G$-conjugate to a subgroup of $Q$ is
centric.  By Lemma \ref{vertices-Lemma} and Lemma \ref{kGi-summands-R} (i), 
if  $i'\in \nu$, then every indecomposable direct summand of
$Bi'\tenkS k$ has  a vertex which is contained  in $Q$ and hence  is  not centric. 
Since $B\tenkS k=$ $Bi\tenkS k\ \oplus\ B(1_B-i)\tenkS k$, it follows that
no indecomposable direct summand of $B(1_B-i)\tenkS k$ has a
centric vertex.  The result follows.
\end{proof}

\begin{proof}[{Proof of Theorem \ref{princ-block-thm}}]
The first statement is in Theorem \ref{Bi-princ-block}, and the second
statement follows  since any essential $p$-subgroup is centric, in
conjunction with \cite[Theorem 1.2]{Rob88}.
\end{proof}

\begin{Remark}
Theorem \ref{Bi-princ-block}  provides an alternative proof of
Theorem \ref{Bi-weights} in the case where $B$ is the principal
block of $kG$, because weight subgroups of the principal block
are centric.
\end{Remark}

As the next example shows, Lemma \ref{non-local-points-2} does not hold in 
general for non-principal  blocks.

\begin{Example}  
Suppose that $p$ is  odd  and let $N$ be a  finite  $p'$-group  satisfying the following:    
$N$    has an automorphism  $\varphi $    of  order $p$  and  $kN$ has an absolutely  
simple and  faithful  module  $V$   of $k$-dimension $d$    with $p-1 \nmid d$    and  
such that  $\,_\varphi V \cong V$.
These hypotheses  are  satisfied   for  instance  if $N$ is an extra-special  $r$ group  
of   order $r^3$  and  exponent  $r$, for $r$  an odd prime   with $p \mid r-1 $ 
(with  $\varphi $  any automorphism  of  order  $p$ acting  trivially on  $Z(N)$  
and    $V$   any faithful  simple  $kN$-module, necessarily of  dimension $r$).   
Let $ b =\frac{1}{N} \sum_{x\in N}  \chi(x) x^{-1} $  be the  central   primitive 
idempotent of  $kN$  corresponding to $V$, where $\chi:  N \to k^{\times} $  
is  the  character of $V$.

Let   $S$  be a finite non-abelian $p$-group   with a maximal abelian subgroup 
$Q$  and let $G  =  N \rtimes S $, with $Q$ acting trivially on $N$ and    some 
generator of $S/Q$ acting as  $\varphi $.    Then    $ C_G(Q)  = N \times Q $,  
$ B= kGe $  is a  block of $kG$  with defect group $S$, and  $(Q, e) $   is a  
centric  $B$-Brauer pair which is normalised by $S$.   Let $\delta $ be the 
unique local  point of $Q$ on $B$ with $\Br_Q(\delta) e \ne 0 $.      
We claim that   $Q_{\delta}$ is  a defect  pointed group of $S_{\nu}$ for 
some  point   $\nu$  of $S$  on $B$.  Indeed, by  Lemma  \ref{non-local-points-1}, 
applied with  $P=S $,   it suffices to prove that the   endopermutation  
$kS/Q$-module  $V_Q  $   contains a nonzero  projective summand.     
Suppose for contradiction  that   $V_Q  $  contains no nonzero  projective 
summands.  Since  $P/Q $ is  cyclic of order $p$, any   non-projective
indecomposable endopermutation   $kP/Q$-module   is  isomorphic to   either  
the trivial  $kP/Q$ module   $k$ or  to   its Heller translate   $\Omega (k) $ of 
dimension $p-1$ (see for instance \cite[Theorem 7.8.1]{LiBookII}).   On the 
other hand,   if $V_Q$   contains  a summand isomorphic to $k$, then  it has 
no summand isomorphic to $\Omega (k)$ (see \cite[Proposition 7.3.7]{LiBookII}). 
Thus, either  $V_Q $  is a direct sum  of  copies  of the trivial module  
$kS/Q$-module or a  direct sum  of copies  of $\Omega (k)$.    Suppose that 
we are in the former case.     Then, $S/Q$ acts trivially on $k(N\times Q/Q )\bar e $,
hence  for all $x \in N$,  $\varphi (x) Q  \bar e  =    x Q \bar e $  for all $ x\in N$.  
Since $Q$ is a  $p$-group  and   $N$  is a  $p'$-group commuting with $Q$, this 
means that  $\varphi (x)  e =    xe $ for all $x \in N$ and   consequently  
$[N, \varphi] $ acts trivially on $V$. This is impossible   since $V$  is  a faithful  
$kN $-module. The latter  case does not occur since    $d$ is not a multiple 
of $p-1 $.
\end{Example}

It is well-known that the principal block idempotent of a finite group algebra 
$kG$ is contained in $\F_p G$. For the sake of completeness, we note that 
one can also choose a source idempotent of the principal block of $kG$ which 
is contained in $(\Fp G)^P$.

\begin{Lemma} \label{Fpsource-Lemma}
Let $G$ be a finite group and $B$ the principal block of $kG$. Let $P$ be a
Sylow $p$-subgroup of $G$ and $i\in $ $B^P$ a source idempotent.
Then, for any field extension $k'/k$, the algebra $B'=$ $k'\tenk B$ is the principal
block of $k'G$, and $i'=$ $1\ten i$ remains a source idempotent of $B'$. 
Equivalently, $B$ has a source idempotent contained in $(\Fp G)^P$. 
\end{Lemma}

\begin{proof}
Since $C_G(P)=Z(P)\times O_{p'}(C_G(P))$ it follows that the principal block of
$kC_G(P)$ is isomorphic to $kZ(P)$. By Brauer's Third Main Theorem
(see e.g. \cite[Theorem 6.3.14]{LiBookII}),  the principal block idempotent of
$kC_G(P)$ is equal to $\Br_P(1_B)$, whence $\Br_P(i)=$ $\Br_P(1_B)$
and $kC_G(P)\Br_P(1_B)\cong$ $kZ(P)$. 
This implies  that $\Br_P(i(kG)^Pi\cong$ $kZ(P)$ is split local,and hence that
$(ikGi)^P$ is split local. Thus $k'\tenk (ikGi)^P\cong$
$(ik'Gi)^P$ is split local, or equivalently,  $i$ remains primitive in
$(k'G)^P$, whence the result. 
\end{proof}

\section{On stable equivalences of Morita type, and
proof of Theorem \ref{splendid-thm}} \label{stable-Section}

Let $k$ be a field. Let $A$, $B$ be finite-dimensional $k$-algebras. 
Following Brou\'e \cite[\S 5.A]{Broue94}, 
an $A$-$B$-bimodule $M$ and a $B$-$A$-bimodule $N$ are said
to {\em induce  a stable equivalence of Morita type between $A$ and $B$} if
$M$, $N$ a both perfect (see Section \ref{backgroundsection}), 
and if we have bimodule isomorphisms 
$M\tenB N\cong$ $A\oplus U$ and 
$N\tenB M\cong$ $B\oplus V$ for some projective $A$-$A$-bimodule $U$ and
some projective $B$-$B$-bimodule $V$. In that case the exact functors
$M\tenB-$ and $N\tenA-$ induces $k$-linear equivalences between  
$\modbar(A)$ and $\modbar(B)$, and if $A$, $B$ are self-injective, then these
are equivalences of triangulated categories.
 See \cite[Sections 2.13, 4.13]{LiBookI} for standard facts and notation
regarding stable categories and equivalences.

We will in particular need the following observation:  if  $A$, $B$ and $C$  
are finite dimensional $k$-algebras and  if $M$ and  $N$ induce a  stable 
equivalence of Morita type between $A$ and $B$, 
then the functors  $M\tenB-$ and $N\tenA -$  induce an equivalence between 
the stable categories $\perfbar(B,C)$ and $\perfbar(A,C)$ of $\perf(B,C)$ and 
$\perf(A,C)$. This goes back to Brou\'e (see the proof of \cite[5.4]{Broue94}); the 
case where $C=A$ is described in detail in \cite[Proposition 2.17.11]{LiBookI}. 
Thus if $M$, $N$ induce a stable equivalence of
Morita type between $A$ and $B$,  and if $V$ is an indecomposable non-projective 
perfect $B$-$C$-bimodule,  then $M\tenA V$ is a direct sum of an indecomposable 
non-projective  perfect $A$-$C$-bimodule and a projective $A$-$C$-bimodule.

If $A$ and $B$ are symmetric, and if $M$, $N$ induce a stable equivalence of 
Morita type between $A$ and $B$, then $M$ and $N$ can be chosen such that $N$ 
is isomorphic to  the $k$-dual $M^\vee$ of $M$; see e.g. 
\cite[Proposition 2.17.5]{LiBookI}.
If the algebras $A$, $B$ are symmetric indecomposable and non-simple as algebras,
then the bimodules $M$ and $N$ can be chosen to be indecomposable 
non-projective and dual to each other. In that case we simply say that $M$ 
induces a stable equivalence of Morita type.

\begin{Theorem} \label{simpleminded-Thm}
Let $A$, $B$ be finite-dimensional self-injective indecomposable non-simple
$k$-algebras such that $A/J(A)$ and $B/J(B)$ are separable.
Let $M$ be an indecomposable $A$-$B$-bimodule and $N$ an indecomposable
$B$-$A$-bimodule such that $M$ and $N$ induce a stable equivalence
of Morita type between $A$ and $B$. Let $V$ be a finitely generated semisimple
$B$-module such that every simple $B$-module is isomorphic to a direct summand
of $V$. Set $U=$ $M\tenB V$.
\begin{itemize}
\item[{\rm (i)}]
Every simple $A$-module is isomorphic to a submodule and a quotient of $U$.
\item[{\rm (ii)}]
The $A$-module  $U$ has $\ell(B)$ isomorphism classes of indecomposable 
direct summands, none of which is projective.
\item[{\rm (iii)}]
If $\End_A(U)$ is self-injective, then $\ell(A)=\ell(B)$.
\end{itemize}
\end{Theorem}

\begin{proof}
The first statement is well-known; we include a short (and standard) argument 
for convenience. Let $X$ be a simple $A$-module. 
Since $M\tenB-$ and $N\tenA-$ induce
equivalences between the stable catgories $\modbar(A)$ and $\modbar(B)$
we have $\Hombar_A(X,U)\cong$ $\Hombar_B(N\tenA X , V)$.
By  Proposition \ref{no-proj-summands-Prop} the $B$-module $N\tenA X$
has no nonzero projective summand. By \cite[Corollary 4.13.4]{LiBookI}
we have $\Hombar_B(N\tenA X,V)\cong$ $\Hom_B(N\tenA X,V)$, so this
space is nonzero, as every simple quotient of $N\tenA X$ is isomorphic to a 
direct summand of $V$. Thus $\Hom_A(X,U)$ is nonzero.
This shows that $X$ is isomorphic to a submodule of $U$, and a similar
argument shows that $\Hom_A(U,X)$ is nonzero, so $X$ is isomorphic to
a quotient of $U$. This proves (i).

By Proposition \ref{no-proj-summands-Prop} the $A$-module $U$ has
no nonzero projective direct summand. Since $U$ and $V$ correspond to each
other via the stable equivalence induced by $M$ and $N$ they have
the same number of isomorphism classes of indecomposable non-projective
summands. Thus $U$ has $\ell(B)$ isomorphism classes of indecomposable 
direct summands. This proves (ii).

Assume now that $\End_A(U)$ is self-injective. It follows from 
Green's theorem  \cite[Theorem 1]{Green78} (see also \cite[Theorem 2.13]{Geck01}) 
that every indecomposable direct summand of $U$ has a 
simple top and socle, and that the isomorphism class of any such summand is 
determined by the isomorphism class of its top, as well as by that of its socle. 
That is, we have bijections with between the isomorphism classes of 
indecomposable direct summands of $U$ and simple $A$-modules 
obtained by sending such a summand to its top (resp. its socle). 
Thus the number of isomorphism classes of idecomposable direct summands
of $U$ is also equal to $\ell(A)$, whence the result.
\end{proof}

For the rest of this section we assume that  $k$  is a field of  prime characteristic $p$.

\medskip
 A  stable  equivalence  between two block algebras   $B$ and $C$ of finite
group  algebras $kG$, $kH$  is called {\em splendid}  if it is a stable equivalence of 
Morita type given  by an indecomposable  trivial source $B$-$C$-bimodule 
$M$ and its dual.  In particular, a  splendid   Morita equivalence between  $B$ and  
$C$  is a  Morita  equivalence given   by an indecomposable  trivial source 
$B$-$C$-bimodule  $M$ and its dual.  The following well-known  result   identifies 
the  bimodules  which induce splendid   stable equivalences.

\begin{Proposition} \label{splendid-stable-source-Prop}
Let $G$, $H$ be finite groups, $B$ a block of $kG$ with a non-trivial defect
group $P$, and $C$ a block of $kH$ with a non-trivial defect group $Q$.
Let $i\in B^P$ and $j\in  C^Q$ be source idempotents.
Let $M$ be an indecomposable $B$-$C$-bimodule inducing a splendid
stable equivalence. There is a group isomorphism $\varphi : P\to Q$
such that $M$ is isomorphic to a direct summand of $Bi\tenkP ({_\varphi(jC)})$.
\end{Proposition}

\begin{proof}
By \cite[Proposition 9.7.1]{LiBookII}, when regarded as a $k(G\times H)$-module,
$M$ has trivial source and a vertex of the form $\{(u, \psi(u))\ |\ u\in P\}$. Thus
$M$ is isomorphic to a direct summand of $\Ind^{G\times H}_{R}(k)\cong$
$kG\tenkP {_\varphi{kH}}$, and since $M$ is a $B$-$C$-bimodule, it is
isomorphic to a direct summand of $B\tenkP {_\varphi{C}}$. 
Since $M$ is indecomposable, it follows  that $M$ is isomorphic to a direct
summand of $Bi'\tenkP {_\varphi{j'C}}$ for some primitive idempotents
$i'\in B^P$ and $j'\in C^Q$. Since $M$ has vertices of order $|P|$ it follows
that $i'$, $j'$ are source idempotents. The local points of $P$ (resp. $Q$)
on $B$ (resp. $C$) are $N_G(P)$-conjugate (resp. $N_H(Q)$-conjugate, and
hence, possibly after replacing $\varphi$, we may assume that $M$ is
isomorphic to a direct summand of $Bi\tenkP {_\varphi{jC}}$ as stated.
\end{proof}

Assuming that $k$ is a splitting field for $B$, $C$ and their Brauer pairs, the 
isomorphism  $\varphi$ in this Lemma 
induces an isomorphism between the fusion systems of  $B$ and $C$ 
on $P$ and $Q$ determined by $i$ and $j$, respectively.  (See  e.g.
\cite[Theorem 9.8.2]{LiBookII}). 

\medskip

In general we do not know what a stable equivalence of Morita type between
two blocks $B$, $C$ does to the source permutation modules. The following 
result shows that a splendid stable equivalence preserves source permutation 
modules, up to projective summands.  

\begin{Theorem} \label{splendid-stable-source}
Let $G$, $H$ be finite groups, $B$ a block of $kG$ with a non-trivial defect
group $P$, and $C$ a block of $kH$ with a non-trivial defect group $Q$.
Let $i\in B^P$ and $j\in C^Q$ be source idempotents. Suppose that $M$ is 
a $B$-$C$-bimodule inducing a splendid stable equivalence.
Then $Bi\tenkP k$ is isomorphic to the non-projective part
of $M\tenC (Cj\tenkQ k)\cong$ $Mj\tenkQ k$. In particular, if $M$ induces a  
splendid Morita equivalence, then  $Bi\tenkP k$ is isomorphic to 
$M\tenC (Cj\tenkQ k)$.
\end{Theorem}

The proof of Theorem \ref{splendid-stable-source} requires the following Lemma.

\begin{Lemma} \label{splendid-source-Lemma}
Let $G$, $H$ be finite groups, $B$ a block of $kG$ with a non-trivial defect group 
$P$ and $C$ a block of $kH$ with a non-trivial defect group $Q$. Let $M$ be an 
indecomposable $B$-$C$-bimodule which induces a stable equivalence of Morita
type. Let $j$ be a source idempotent in $C^Q$.  
\begin{itemize}
\item[{\rm (i)}]
The $B$-$kQ$-bimodule $Mj$ has, up to isomorphism, a unique non-projective 
indecomposable direct summand.
\item[{\rm (ii)}]
 If $M$ induces a splendid stable equivalence, then the up to isomorphism
 unique non-projective indecomposable direct summand of $Mj$ is isomorphic 
 to $(Bi)_{\psi}$ for some source idempotent $i\in$ $B^P$ and some group 
 isomorphism  $\psi : Q\to P$. 
 \end{itemize}
\end{Lemma}

\begin{proof}
The $C$-$kP$-bimodule $Cj$ is indecomposable since $j$ is primitive in $C^P$,
and it is non-projective since $Q$ is non-trivial. As $M$ induces a stable
equivalence of Morita type, it follows from an  observation  at the beginning of
this Section  that the $B$-$kQ$-bimodule $Mj\cong$ 
$M\tenC Cj$ is a direct sum of an
indecomposable non-projective $B$-$kQ$-module and a projective 
$B$-$kQ$-module. This shows the first statement. 
Suppose now that $M$ is a  $p$-permutation module.
By  Proposition \ref{splendid-stable-source-Prop}, $M$ is isomorphic to a direct 
summand of 
$Bi\tenkP (_\varphi{(jC)})$ for some source idempotent $i\in$ $B^P$ and some
isomorphism $\varphi : P\to Q$. Thus the $B$-$kQ$-bimodule $Mj$ is isomorphic 
to a direct summand of $Bi\tenkP (_\varphi(jCj))$. By Lemma 
\ref{source-bimod-Lemma}, in the decomposition  of  $jCj$ as a 
$kQ$-$kQ$-bimodule  the summands (now regarded as  $k(Q\times Q)$-modules) 
with a vertex of order $|Q|$  are isomorphic to $(kQ)_\tau$, with $\tau$ an 
automorphisms of $Q$ and there is at least one such summand. 
Note that $M^\vee\tenB Mj$ has a summand isomorphic 
to $Cj$ as a $C$-$kQ$-module, which has a vertex $\Delta(Q)$.
Using Lemma \ref{perfect-vertices-Lemma} and Lemma 
\ref{bimod-idempotent-Lemma} shows  that $Mj$ has vertices of order $|P|=|Q|$.
It follows that the  non-projective summand of $Mj$ is isomorphic to a direct 
summand of $Bi\tenkP (_\mu kQ)$ for some isomorphism $\mu : P\to Q$ (obtained 
from composing $\varphi$ with an automorphism of $Q$). This last bimodule
is isomorphic to $(Bi)_\psi$, where $\psi=\mu^{-1}$, whence  the  second statement.  
The last statement is immediate  from the second  since  if $M$ induces a  Morita 
equivalence, then the  fact that   $(Cj\tenkQ k)$ has no nonzero projective 
summands implies  that  $M\tenC (Cj\tenkQ k)$  has no  nonzero projective summand.
\end{proof}

\begin{proof}[{Proof of Theorem \ref{splendid-stable-source}}]
By  Lemma \ref {isomorphic-summands-Lemma},
 the isomorphism class of $Bi\tenkP k$ is independent
of the choice of a defect group $P$ and source idempotent $i$.
Since $M$ induces a splendid stable equivalence, we may assume, by 
Lemma \ref{splendid-source-Lemma}, that $Mj\cong$ $(Bi)_\psi \oplus Y$ 
for some projective $B$-$kQ$-bimodule $Y$ and some
group isomorphism $\psi : Q\to P$. Now $(Bi)_\psi \tenkQ k\cong$
$Bi\tenkP k$, and this $B$-module has no nonzero projective direct
summand (cf. Proposition \ref{no-proj-summands-Prop}) 
while $Y\tenkQ k$ is a projective $B$-module.
Thus $Bi\tenkP k$ is isomorphic to the non-projective part of $Mj\tenkQ k$,
whence the result. 
\end{proof}

For $C$ a Brauer correspondent of $B$ this yields the following more  
precise  result.

\begin{Corollary}\label{splendid-source-Cor}
Let $G$ be a finite group, $B$  a block of $kG$ with a non-trivial defect group 
$P$. Let $C$ be the block of $kN_G(P)$ which is the Brauer correspondent of 
$B$. Let $j\in C^P$ and $i\in B^P$ be source idempotents. Suppose that $k$ is
a splitting field  for $B$, $C$ and their Brauer pairs.
Let $M$ be an indecomposable perfect $B$-$C$-bimodule inducing a splendid 
stable equivalence. Then there is an automorphism $\psi$ of $P$ such that
$Mj\cong$ $Bi_\psi$ as $B$-$kP$-bimodules. In particular, the functor
$M\tenC-$ sends the soure permutation $C$-module $Cj\tenkP k$ to a
module isomorphic to the source permutation $B$-module $Bi\tenkP k$.
\end{Corollary}

\begin{proof}
By Proposition \ref{normal-Prop} and the assumptions on $k$,
the source permutation $C$-module $Cj\tenkP k$ is split semisimple, 
with $\ell(C)=w(C)$ isomorphism classes of simple summands.
 Proposition \ref{no-proj-summands-Prop} implies that $M\tenC Cj\tenkP k$
 has no nonzero projective summand as a left $B$-module. 
 The result follows from Theorem \ref{splendid-stable-source}.
 \end{proof}

Our next task is to  give a proof of Theorem  \ref{splendid-thm}.

\begin{Theorem} \label{stable-Brauer-corr}
Let $G$ be a finite group, $B$ a block of $kG$ with a non-trivial defect group $P$,
and let $C$ be  the block of $kN_G(P)$ which is  the Brauer correspondent of $B$.
Suppose that $k$ is a splitting field  for the blocks $B$, $C$ and their Brauer pairs.
Let $j\in$ $C^P$ be a source idempotent.
Let $M$ be an indecomposable $B$-$C$-bimodule which induces a stable
equivalence of Morita type. Assume that the algebra $\End_B(Mj\tenkP k)$ is 
self-injective. The following hold.
\begin{itemize}
\item[{\rm (i)}] 
We have  $\ell(B)=\ell(C)$. 
\item[{\rm (ii)}]
If $M$ induces a splendid stable  equivalence  between $B$ and $C$, then 
$\ell(B)=w(B)$; that is,  Alperin's weight conjecture holds for $B$.
\end{itemize}
\end{Theorem}

The reason why  we are not stating in (i)   that $B$ satisfies Alperin's weight conjecture
is that for an arbitrary stable equivalence of Morita type between $B$ and its Brauer 
correspondent $C$ we do not know whether  this is sufficient for $B$ 
and $C$ to have  the same fusion systems - we only have 
that $\ell(C)=w(C)$. If $P$ is abelian, or more generally, if $N_G(P)$ controls 
$B$-fusion, then we do get the equality $\ell(C)=w(B)$. The second statement 
of the above Theorem makes use of the fact  that this is the  case if $M$ induces a 
splendid stable  equivalence.

\begin{proof}[{Proof of Theorem \ref{stable-Brauer-corr}}]
By Proposition \ref{normal-Prop} the $C$-module $Cj\tenkP k$ is semisimple,
with $\ell(C)$ isomorphism classes of simple summands. We have
$M\ten_C (Cj\tenkP k)\cong$ $Mj\tenkP k$. 
It follows from Theorem \ref{simpleminded-Thm}, applied
to $B$, $C$, the $B$-$C$-bimodule  $M$ and the semisimple $C$-module 
$V=Cj\tenkP k$ that if $\End_B(Mj\tenkP k)$ is self-injective, then $\ell(B)=$ 
$\ell(C)$. This shows (i).

A maximal $B$-Brauer pair of the form $(P,e)$ for some block $e$ of $kC_G(P)$
 is also a  maximal $C$-Brauer pair.
Thus any $C$-weight is a $B$-weight. If we denote by $\CF$, $\CG$ the
fusion systems of $B$, $C$, respectively, associated with $(P,e)$,
then $\CG = N_\CF(P)$ (see e.g. \cite[Proposition 8.5.4]{LiBookII}).
If $M$ is splendid, or equivalently, if $M$ is a trivial source module  when
regarded as a $k(G\times H)$-module, then it follows from 
\cite[Theorem 9.8.2]{LiBookII} that $\CF={^\varphi{\CG}}$ for some
automorphism $\varphi$ of $P$. Since $\CF$ contains $\CG$ this implies 
$\CF=\CG$. 
In particular, $P$ is the only weight subgroup for both $B$ and $C$, and so 
$\ell(C)=w(C)=w(B)$ in that case, which implies the second statement.
\end{proof}

\begin{proof}[{Proof of \ref{splendid-thm}}]
Theorem \ref{splendid-thm} follows from combining the two statements
in Theorem \ref{stable-Brauer-corr}.
\end{proof}

A stable equivalence of Morita type between a block $B$ and a Brauer
correspondent $C$ is frequently given by truncated induction/restriction.
This is a situation which arises, for instance, for principal blocks with
cyclic defect, for blocks of defect $1$, and for blocks with Klein four
defect groups - cases that appear in the main results in the Introduction.
We include some characterisations of this situation.

\begin{Proposition}  \label{splendid-stable-Prop}
Let $G$ be a finite group and $B$ a block of $kG$ with a non-trivial defect group $P$.
Let $C$ be the block of $kN_G(P)$ with $P$ as defect group which is the Brauer 
correspondent of $B$.  Write $B=kGb$ for $b$ a primitive idempotent in $Z(kG)$
and $C=kN_G(P)c$ for some primitive idempotent $c$ in $Z(kN_G(P))$.
The following are equivalent.
\begin{itemize}
\item[{\rm (i)}]
The $B$-$C$-bimodule $Bc$ and its dual, which is isomorphic to $cB$,
induce a stable equivalence of Morita type between $B$ and $C$.
\item[{\rm (ii)}]
As a $C$-$C$-bimodule we have $cBc\cong$ $D\oplus Y$ for some projective
$D$-$D$-bimodule.
\item[{\rm (iii)}]
For any source idempotents $i\in B^P$ and $j\in $ $C^P$, 
the $iBi$-$jDj$-bimodule $iBj$ and its dual $jBi$ 
induce a stable equivalence of Morita type between $iBi$ and $jDj$.
\item[{\rm (iv)}]
For some source idempotents $i\in B^P$ and $j\in D^P$, 
the $iBi$-$jDj$-bimodule $iBj$ and its dual $jBi$ 
induce a stable equivalence of Morita type between $iBi$ and $jDj$.
\item[{\rm (v)}]
For any source idempotent $j\in C^P$ there is a source idempotent
$i\in B^P$ which is contained in $jBj$ and which commutes with $jCj$.
For any such choice of $i$ and $j$, multiplication by $i$ induces an
algebra homomorphism $jCj\to iBi$ which is split injective as
a $jCj$-$jCj$-bimodule homomorphism such that
we have an isomorphism of $jCj$-$jCj$-bimodules
$iBi \cong$ $jCj\oplus Z$ for some projective $jCj$-$jCj$-bimodule $Z$. 
\end{itemize}
If these equivalent conditions hold, then the $B$-$C$-bimodule $Bc$ has,
up to isomorphism, a unique indecomposable non-projective direct 
summand $M$, and then $M$ induces a splendid stable equivalence
between $B$ and $C$, and if $k$ is large enough, then
we have $Bi\tenkP k\cong$ $Mj\tenkP k\cong$ $M\tenC Cj\tenkP k$.
\end{Proposition}

\begin{proof}
The implication (i) $\Rightarrow$ (ii) is proved in \cite[Proposition 9.8.3]{LiBookII}.
Noting that $cBc\cong$ $cB\tenB Bc$, it follows that the converse implication 
(ii) $\Rightarrow$ (i) is just part of the definition of
a stable equivalence of Morita type. The equivalence of (i), (iii), and (iv) follows
from the standard Morita equivalence between a block and any of its source
algebras; this is spelled out in \cite[Theorem 9.8.2]{LiBookII}.
Since the local points of $P$ on $B$ (resp. $c$) are $N_G(P)$-conjugate it follows that
if statement (iii) holds, then statement (iii) holds with $i$, $j$ replaced by any source 
idempotents.  
By \cite[Theorem 6.15.1]{LiBookII}, for any source idempotent
$j\in C^P$ there is a source idempotent $i\in B^P$ contained in $jBj$, such that
$i$ commutes with $jCj$ and such that multiplication by $i$ induces an
algebra homomorphism $jCj\to iBi$ which splits as a homomorphism of
$jCj$-bimodules. It follows from \cite[Proposition 9.8.4]{LiBookII} that
$iBj$ and its dual induce a stable equivalence of Morita type if and only
if the image of $jCj$ in $iBi$ has a complement as a $jCj$-$jCj$-bimodule
which is projective. This proves the equivalence of (iv) and (v).
It follows from \cite[2.4]{Listable} (or \cite[Theorem 4.14.2]{LiBookI}) 
that $Bc$ has up to isomorphism a unique indecomposable
non-projective  direct summand $M$ as a $B$-$C$-bimodule, and that this 
summand induces a stable equivalence of Morita type. Since $M$ is a
summand of the $p$-permutation bimodule  $Bc$, it follows that $M$ induces
a splendid stable equivalence. The last statement follows from
Corollary \ref{splendid-source-Cor}.
\end{proof}

\section{Brauer tree algebras.} 
\label{Brauertree-section} 

The first result in this Section illustrates that Brauer tree algebras provide 
examples for the scenario considered in Theorem \ref{simpleminded-Thm}.
Let $k$ be a field.

\begin{Theorem} \label{Brauertree-thm}
Let $A$ be a non-simple Brauer tree algebra over a field $k$. Let $\{U_i\}_{i\in I}$ 
be a  family of indecomposable non-projective $A$-modules which correspond
to a set of representatives of the isomorphism classes of simple $B$-modules
under a stable equivalence between $A$ and a symmetric Nakayama algebra $B$.
Suppose that $k$ is a splitting field for $A$ and $B$.
Then the algebra $\End_A(\oplus_{i\in I}\ U_i)$ is a direct product of
self-injective Nakayama algebras; in particular, $\End_A(\oplus_{i\in I}\ U_i)$
is self-injective.
\end{Theorem} 

The Heller operator $\Omega$ on the stable cateogry $\modbar(A)$
is an auto-equivalence, and hence the above Theorem applies to the Heller 
translates of the $U_i$ as well.  Thus the algebra
$$\End_A(\oplus_{i\in I}\ \Omega_A(U_i))$$ 
is also a direct product of
self-injective Nakayama algebras, hence self-injective. The two endomorphism
algebras of $\oplus_{i\in I} \ U_i$ and $\oplus_{i\in I}\ \Omega_A(U_i)$
need not be Morita equivalent, however, and need not even have the same
number of blocks, as we will see in examples in Section \ref{Example-section}.

\medskip
The  material needed for the proof of Theorem
\ref{Brauertree-thm} goes back to \cite{GaRie}, \cite{Greenwalk}; we follow
the exposition and notation from \cite[Sections 11.6, 11.7]{LiBookII}.

\medskip
A finite-dimensional $k$-algebra $B$ is called a {\em Nakayama algebra} if all
projective indecomposable $B$-module are uniserial (that is, have a unique
composition series).  An indecomposable self-injective split 
Nakayama algebra $B$ has a cycle as quiver such that each path of a fixed length 
$d\geq 1$ starting at any vertex is a zero relation; that is, 
all projective indecomposable modules have the same composition length $d$.
In that case, if $\{T_j\}_{j\in J}$ is a set of representatives of the isomorphism classes of
simple $B$-modules, there is a unique cyclic transitive permutation $\rho$ of $J$
such that the unique composition series of a projective cover of $T_j$ has 
composition factors, from top to bottom, isomorphic to 
$$T_j, T_{\rho(j)}, T_{\rho^2(j)},\dots, T_{\rho^{d-1}}(j).$$
For $B$ to be moreover symmetric, we need that the top and bottom composition 
factors are isomorphic; that is, $T_j\cong$ $T_{\rho^{d-1}(j)}$,
and hence $d-1$ is  a multiple of the number $\ell(B)=|J|$ of isomorphism classes
of simple $B$-modules. If $B$ is an indecomposable split symmetric Nakayama
algebra, then $B$ is a Brauer tree algebra, with a star having $\ell(B)$ edges all 
emanating from a vertex to which the exceptional multiplicity $m=\frac{d-1}{|J|}$ 
is attached, where $d$ is as before  the common length of the projective indecomposable
$B$-modules. 

\medskip
Let $A$, $B$ be finite-dimensional indecomposable split symmetric non-simple 
$k$-algebras  such that $B$ is a Nakayama algebra and such that there is a 
$k$-linear stable equivalence $\CF : \modbar(A)\cong\modbar(B)$.
We index simple modules and projective indecomposable modules by  primitive 
idempotents: choose sets of representatives $I$ and $J$ of the conjugacy classes of 
primitive idempotents in $A$ and $B$, respectively. 
Thus, setting $S_i=Ai/J(A)i$ and $T_j=Bj/J(B)j$, the sets $\{S_i\}_{i\in I}$ and
$\{T_j\}_{j\in J}$ are sets of representatives of the isomorphism classes of
simple $A$-modules and simple $B$-modules, respectively. By 
\cite[Proposition 11.6.2]{LiBookII}, $A$ is split if and only if $B$ is split.

\medskip
By \cite[Theorem 11.6.3]{LiBookII}, each projective indecomposable $A$-module
$Ai$, where $i\in I$, has two unique uniserial submodules $U_i$ and $V_i$ such that
$\CF(U_i)$ and $\Omega_B(\CF(V_i))$ are simple $B$-modules. More precisely,
{\em any} $k$-linear stable equivalence between $A$ and $B$ sends the family 
$(T_j)_{j\in J}$ to either the family $(U_i)_{i\in I}$ or the family $(V_i)_{i\in I}$, 
up to isomorphism. The two submodules $U_i$, $V_i$ of $Ai$ satisfy 
$U_i\cap V_i=\soc(Ai)$ and $U_i+V_i=J(A)i$.
Moreover, $U_i/\soc(Ai)$ and $V_i/\soc(Ai)$ have no common composition factors.
By \cite[Theorem 11.6.4]{LiBookII}, there are unique permutations $\rho$ and $\sigma$
such that the composition series of $U_i$ and $V_i$,  from top to bottom, are
$$S_{\rho(i)}, S_{\rho^2(i)},...., S_{\rho^a(i)}=S_i$$
$$S_{\sigma(i)}, S_{\sigma^2(i)},...., S_{\sigma^b(i)}=S_i . $$
The integers $a$ and $b$ are equal to the composition lengths of $U_i$ and $V_i$,
respectively, and because the bottom composition factor of either module
is $S_i$, the integer $a$ is a multiple of the length of the 
$\langle\rho\rangle$-orbit $i^\rho$ of $i$,and the integer $b$ is a multiple of the 
length of the $\langle\sigma\rangle$-orbit $i^\sigma$ of $i$. 
The integers $a$ and $b$ depend only on the orbits $i^\rho$ and $i^\sigma$ of $i$,
respectively. Since $U_i/\soc(Ai)$ and $V_i/\soc(Ai)$ have no common
composition factors, it follows that at least one of $a$ or $b$ is in fact equal to
the length $|i^\rho$| of $i^\rho$  and the length $|i^\sigma|$ of $i^\sigma$,
respectively, and that $i$ is
the unique element of $I$ in the intersection $i^\rho\cap i^\sigma$
of the $\langle\rho\rangle$-orbit and the
$\langle\sigma\rangle$-orbit of $i$. 

\medskip
By \cite[Theorem 11.7.2]{LiBookII}, except for possibly one of these orbits,
the integers $a$ and $b$ are equal to the lengths of the orbits of $i$ under
$\langle\rho\rangle$ and $\langle\sigma\rangle$, respectively. 
That is, there is at most one orbit of either
$\rho$ or $\sigma$ for which there is an exceptional multiplicity $m>1$
such that $U_i$ or $V_i$ has composition factors with multiplicity $m$.

\medskip
The permutations $\rho$ and $\sigma$ determine the Brauer tree as follows:
the vertices of the Brauer tree are the orbits $i^\rho$ and $i^\sigma$,with $i\in I$,
and we draw an edge labelled $i$
between $i^\rho$ and $i^\sigma$. Since $i^\rho\cap i^\sigma=$ $\{i\}$, there are
no multiple edges between any two such orbits, and more precisely, this yields a tree,
with possibly one exceptional orbit to which we attach the exceptional multiplicity
$m$,  if this case arises.  The composition $\rho\circ\sigma$ is a transitive cycle on $I$.
If $A$ is itself a Nakayama algebra, then one of the two 
permutations is the identity and the other is a transitive cycle on $I$.
We prove the following slightly more precise version of
Theorem \ref{Brauertree-thm}.

\begin{Theorem} \label{Brauertree-thm-2}
With the notation above,  the algebra $\End_A(\oplus_{i\in I}\ U_i)$ is a direct 
product of algebras $N_R$, where $R$ runs over the set of
$\langle\rho\rangle$-orbits in $I$, and where $N_R$ is a self-injective Nakayama algebra
such that $\ell(N_R)=$ $|R|$ and such that the projective indecomposable
$N_R$-modules have composition length $|R|$ if $R$ is non-exceptional and
composition length $m|R|$ if $R$ is exceptional.
\end{Theorem}

\begin{proof}
We keep the notation above. 
If $i$ and $i'$ in $I$ belong to two different $\rho$-orbits, then $U_i$ and $U_{i'}$ have 
no common composition factor, so $\Hom_A(U_i,U_{i'})$ and $\Hom_A(U_{i'}, U_i)$ are
both zero. Thus the endomorphism algebra $\End_A(\oplus_{i\in I} U_i)$ is a direct
product of the algebras $\End_A(\oplus_{s=1}^{a_\rho} U_{\rho^s(i)})$, where
$a_\rho$ is the length of the $\rho$-orbit in $I$, and where $i$ runs over a set
of representatives of the $\rho$-orbits in $I$.

We will show next that there is an $A$-homomorphism $U_{\rho(i)}\to U_i$
with simple kernel and cokernel.
Since $S_{\rho(i)}$ is the top composition factor of $U_i$, there is a surjective
$A$-homomorphism $A_{\rho(i)}\to U_i$. Precomposing this with the
inclusion $U_{\rho(i)}\to A\rho(i)$ yields an $A$-homomorphism
$$\alpha_i : U_{\rho(i)} \to U_i$$
and by comparing composition factors, the image of $\alpha_i$ is the unique
maximal submodule of $U_i$ and the kernel is the unique simple submodule
$\soc(U_{\rho(i)})=$ $\soc(A\rho(i))$ of $U_{\rho(i)}$.

All homomorphisms between the $U_{i'}$ with $i'$ running over the 
$\rho$-orbit of $i$ are linear combinations of compositions of the above.
Thus the quiver of $\End_A(\oplus_{i}\ U_i)$, with $i$ running over a $\rho$-orbit
$R$, is a cycle with $a_\rho=$ $|R|$ edges, and and any full cycle becomes zero if the
$\rho$-orbit is non-exceptional, or $m$-full cycles become zero if the $\rho$-orbit
is exceptional with exceptional multiplicity $m$.
 
It follows that the algebra $\End_A(\oplus_{i}\ U_i)$, with $i$ running over a
$\rho$-orbit $R$, is a split basic  self-injective Nakayama algebra with $a_\rho=$
$|R|$  isomorphism 
classes of simple modules and projective indecomposable modules of 
length $|R| m$ or $|R| $, depending on whether the $\rho$-orbit $R$ of $i$
is exceptional or not.
This proves Theorem \ref{Brauertree-thm-2}.
\end{proof}

Theorem \ref{Brauertree-thm} follows obviously from Theorem \ref{Brauertree-thm-2},
with the added information that if $\rho$ is not the identity, then 
$\End_A(\oplus_{i\in I} U_i)$ is not symmetric, since the Nakayama algebras $N_R$
for $R$ a non-trivial orbit cannot be symmetric as the socles of some projective
indecomposable modules are not  isomorphic to the top composition factors
of these. 

\begin{Remark} \label{Omega-remark}
With the notation above, any self stable equivalence of $\modbar(B)$ coincides
on modules with some power of the Heller operator $\Omega_B$.
 Moreover,
$\Omega^2_B$ permutes the isomorphism classes of simple $B$-modules.
Since Heller operators commute with stable equivalences on modules, it follows
that $\Omega_A^2$ permutes the isomorphism classes of the modules $U_i$, 
$i\in I$. As briefly mentioned in the Introduction,
Theorem \ref{Brauertree-thm} using the stable equivalence 
$\CF\circ \Omega_B$ implies that the algebra $\End_A(\oplus_{i\in I}\ \Omega_A(U_i))$
is also self-injective. This algebra is in fact isomorphic to $\End_A(\oplus_{i\in I}\ V_i)$
because $\Omega_A(U_i)\cong$ $V_{\rho(i)}$. 
The passage from $\CF$ to $\CF\circ\Omega_B$ amounts to
exchanging the roles of the modules $U_i$, $V_i$ and of the permutations 
$\rho$, $\sigma$.
Since the numbers of $\rho$-orbits and of $\sigma$-orbits may differ  (by at most $1$),
the numbers of blocks of $\End_A(\oplus_{i\in I}\ U_i)$ and of
 $\End_A(\oplus_{i\in I}\ V_i)$ may differ as well by at most $1$.
 Since the $U_i$ and $V_i$ correspond to pairwise non-isomorphic simple modules
 under some stable equivalence, both $\Endbar_A(\oplus_{i\in I}\ U_i)$ and
 $\Endbar_A(\oplus_{i\in I}\ V_i)$ are isomorphic to the direct product $k^{|I|}$
 of $|I|$ copies of $k$. Thus the ideal $\End_A^\pr(\oplus_{i\in I}\ U_i)$ of 
 endomorphisms which factor through a projective $A$-module  is equal to the
 radical of $\End_A(\oplus_{i\in I}\ U_i)$, and the analogous statement holds for
 $\End_A(\oplus_{i\in I}\ V_i)$.
 \end{Remark} 
 
 \section{Cyclic blocks, and proof of Theorem \ref{cyclic-thm}} \label{cyclic-Section}
 
 Let $k$ be a field of prime characteristic $p$.
 Let $G$ be a finite group and $B$ a block of $kG$ with a non-trivial cyclic defect
 group $P$. Denote by $Z$ the unique subgroup of order $p$ of $P$.
 Via the Brauer correspondence, the block $B$ determines unique
 blocks $C$ of $kN_G(Z)$ and $D$ of $kN_G(P)$ (for the uniqueness of $C$, see
 \cite[Theorem 9.8.6]{LiBookII}). Induction and restriction
 between $kG$ and $kN_G(Z)$ followed by the truncation with the block
 idempotents $b$ of $B$ and $c$ of $C$ yields a stable equivalence of Morita type
 between $B$ and $C$. If $k$ is a splitting field for $B$ (and hence, by
 \cite[Proposition 11.6.2]{LiBookII}, for $C$), then, by \cite[Theorem 11.2.2]{LiBookI}, 
 the algebras $C$ and $D$ are Morita equivalent to
 the symmetric Nakayama algebra $k(P\rtimes E)$, where $E$ is a cyclic
 subgroup of $\Aut(P)$ of order dividing $p-1$ (this is the inertial quotient of $B$).
 Thus $B$ is Morita equivalent to a Brauer tree algebra. Although the blocks
 $C$ and $D$ are Morita equivalent, a Morita equivalence between these two 
 algebras is not, in general, induced by induction/restriction functors between 
 $N_G(Z)$ and $N_G(P)$.
 This is the reason for the extra hypotheses in Theorem \ref{cyclic-thm}, which
 ensure that  induction/restriction between $G$ and $N_G(P)$, truncated at the 
 block idempotents,  
 induces a stable equivalence of Morita type between $B$ and $D$.
 
 \begin{Proposition} \label{splendid-stable-cyclic-Prop}
 Let $G$ be a finite group, $B$ a block of $kG$ with a non-trivial cyclic defect
 group $P$, and let $D$ be the block of $kN_G(P)$ with defect group $P$
 which is the Brauer correspondent of $B$. Assume that $k$ is a splitting field
 for $B$ and $D$ and their Brauer pairs. 
 Let $d$ be the block idempotent of $D$.   Let $Z$ be subgroup of order $p$
 of $P$ and let $C$ be the unique block of $kN_G(Z)$ corresponding to $B$.
 Denote by $c$ the block idempotent of $C$ in $Z(kN_G(Z))$.
 The following are equivalent.
 \begin{itemize}
 \item[{\rm (i)}] The $B$-$D$-bimodule $Bd$ and its dual induce a stable 
 equivalence of Morita type between $B$ and $D$.
 \item[{\rm (ii)}] The $C$-$D$-bimodule $Cd$ and its dual induce a stable 
 equivalence of Morita type between $C$ and $D$.
 \item[{\rm (iii)}] The simple $C$-modules have endotrivial sources.
 \item[{\rm (iv)}] The simple $kC_G(Z)c$-modules have endotrivial sources.
 \item[{\rm (v)}] The $kC_G(Z)c$-$kP$-bimodule $kC_G(Z)j$ and its dual induce
 a stable equivalence of Morita type between $kC_G(Z)f$ and $kP$, 
 where $j$ is a source idempotent in $(kC_G(Z)c)^P$.
 \end{itemize}
 These equivalent statements hold in particular if $P$ has order $p$,  or
 if $B$ is the principal block of $kG$.
 \end{Proposition}
 
 \begin{proof}
 Note that $N_G(P)\subseteq $ $N_G(Z)$, and hence
 $d$ commutes with $c$. 
 By standard facts (e. g. \cite[Theorem 9.8.6]{LiBookII}) 
 the $B$-$C$-bimodule $bkGc=$ $Bc$ and its dual induce a stable equivalence of
 Morita type between $B$ and $C$.  Note further that $c$ is contained in 
 $(kC_G(Z))^{N_G(Z)}$, hence  a sum of block idempotents of $kC_G(Z)$; choose
 a block $f$ of $kC_G(Z)$ satisfying $fc=f$.
 
 Since stable equivalences of Morita type
 can be composed,  the equivalence of (i) and (ii) follows. 
 Note that $B$ and $C$ have the same Brauer correspondent.
 The blocks $C$ and $D$ have source algebras of the form
 $S\tenk k(P\rtimes E)$ and $k(P\rtimes E)$, where $E$ is the inertial
 quotient of $B$, $C$, $D$ (which is a cyclic group of order dividing $p-1$),
 and where $S=\End_k(V)$ for some indecomposable endopermutation 
 $kP$-module with vertex $P$.
 Note that $V$ is the source of all simple $C$-modules as well as the
 unique (up to isomorphism) simple $kC_G(Z)f$-module, where $f$ is a block
 of $kC_G(Z)f$ appearing in the decomposition of $c$ in $Z(kC_G(Z))$.
 This shows the equivalence of (iii) and (iv). 
 
The source algebras for $kC_G(Z)f$ have the form $S\tenk kP$.
As a $k(P\times P)$-module, we have $S\tenk kP\cong$ 
$S\tenk \Ind^{P\times P}_{\Delta P}(k)$, where $\Delta P=\{(u,u)\ |\ u\in P\}$.
By standard reciprocity statements (e. g. \cite[Proposition 2.4.12]{LiBookI}),
this module is isomorphic to $\Ind^{P\times P}_{\Delta P}(S)$. The action
of $\Delta P$ on $S$ is the conjugation action on $S$, so this is isomorphic to
$\Ind^{P\times P}_{\Delta P}(V\tenk V^*)$, where $V^*$ is the $k$-dual of $V$. 
Since $V$ is an indecomposable
endopermutation $kP$-module with vertex $P$, we have an isomorphism
of $k\Delta P$-modules $V\tenk V^*\cong$ $k \oplus T$ for some
permutation $k\Delta P$-module $T$ with no trivial summand. Moreover,
$V$ is endotrivial if and only if $T$ is projective. Thus $S\tenkP kP$ is
isomorphic to $kP$ plus a projective $kP$-$kP$-bimodule if and only
if $V$ is endotrivial.
Using Proposition \ref{splendid-stable-Prop}, this  implies the equivalence of 
(iv) and (v).
 
By \cite[Theorem 11.2.2]{LiBookII}, a source idempotent in $(kC_G(Z)f)^P$ 
remains a source idempotent in  $C^P$.
The $kP$-module structure on $V$ extends to $k(P\rtimes E)$, through which
we obtain an injective algebra homomorphism $k(P\rtimes E)\to$
$S\tenk k(P\rtimes E)$. By Proposition \ref{splendid-stable-Prop}, the
bimodule $Cd$ induces a stable equivalence of Morita type between $C$ and $D$
if and only if the image of this homomorphism  has a projective complement as
a bimodule for $k(P\rtimes E)$. Using arguments as before (with $P\rtimes E$ 
instead of $P$) one sees that this happens if and only if $V$ is endotrivial.
This shows the equivalence of (ii) and (iii).
For the last statement we note that if $B$ is the principal block, then
the $kP$-module $V$ as defined before is the trivial $kP$-module, so 
in particular endotrivial. If $P$ has order $p$, then $V$ is a Heller translate of
the trivial $kP$-module, so also endotrivial. The last statement follows.
\end{proof}

 \begin{proof}[{Proof of Theorem \ref{cyclic-thm}}]
 With the notation of Theorem \ref{cyclic-thm}, 
 let $f$ be a primitive idempotent in $B^{N_G(P)}$ such that in a decomposition
 of $f$ in $B^P$ there is a source idempotent of $B$. Let $j\in D^P$ be a source
 idempotent. Then $i=fj$ is a source idempotent in $B^P$, and we have
 $jDj\cong k(P\rtimes E)$. Multiplication by $f$ embeds $k(P\rtimes E)$ into
 the source algebra $A=iBi$ in such a way, that the induction/restriction functors
 between $A$ and $k(P\rtimes E)$ induce a stable equivalence of Morita type
 (and this is where we use the extra hypotheses on $B$, which ensure this.)
 Note that $A\tenkP k\cong$ $A\ten_{k(P\rtimes E)} kE$ and that $kE$ is the
 direct sum of a set of representatives of the isomorphism classes of simple
 $k(P\rtimes E)$-modules. Thus $A\tenkP k=$  $\oplus_{i\in I}\ U_i$, where
 the modules $U_i$ are now as in Theorem \ref{Brauertree-thm}.
 Since $\End_{kG}(B i \tenkP k)\cong$ $\End_A(A\tenkP k)\cong$ 
 $\End_A(A\ten_{k(P\rtimes E)} kE)$, Theorem \ref{cyclic-thm} is a 
 consequence of Theorem \ref{Brauertree-thm}.
 \end{proof}
  
  We record  the  following result for future reference.  
  Recall that if a finite group $G$   has  a cyclic Sylow $p$-subgroup $P$, then 
  $N_G(P)/C_G(P)$ acts freely on $P\setminus \{1\}$.
  
 \begin{Proposition} \label{proj-free-Frob-Prop}
 Let $G$ be a finite group with a non-trivial abelian Sylow $p$-subgroup $P$, let
 $B$ be the principal block of $kG$, and let $i$ be a source idempotent in $B^P$.
 Suppose that $N_G(P)/C_G(P)$ acts freely on $P\setminus \{1\}$.
 Set $i' = b-i$. Then $Bi'$ is projective as a $B$-$kP$-bimodule, and $Bi'\tenkP k$
 is projective as a left $B$-module. Equivalently, $Bi\tenkP k$ is the non-projective
 part  of the left $B$-module $B\tenkP k$.   Further,  if  $P$  is  cyclic, then  $k$ is  a 
 splitting field  for  $B$, its  Brauer correspondent in $N_G(P)$ and  their Brauer pairs.
 \end{Proposition}
 
\begin{proof}
Denote by $b$ the block idempotent of the principal block.  
By Brauer's Third Main Theorem (see e.g. 
\cite[Theorem 6.3.14]{LiBookII}),  for any subgroup $Q$ of $P$, $\Br_Q(b)$ is 
the principal block idempotent of  $kC_G(Q)$.  For $Q\neq 1$ 
the block  $\Br_Q(b)$ is  nilpotent  (cf. \cite[Lemma 10.5.2]{LiBookII}), and hence
$kC_G(Q)\Br_Q(b)\cong kP$, by \cite[Corollary 8.11.6]{LiBookII}. 
Since $\Br_Q(i)\neq 0$ for any subgroup $Q$
of $P$, and since $1$ is the unique idempotent in $kP$, it follows that
$\Br_Q(i)=1$ and hence that $\Br_Q(i')=0$ for all non-trivial subgroups $Q$ of 
$P$. But then $i'$ is contained in the space $B_1^P$ of traces from $1$ to $P$. 
Thus $Bi'$ is projective by  Lemma  \ref{bimod-idempotent-Lemma} (ii).  
The  first assertion  follows since  by  Lemma~\ref{no-proj-summands-Lemma} 
$Bi\tenkP k$  contains no nonzero projective summands.   Since ${\mathbb F}_p$ 
is a splitting field for any principal  block  with cyclic defect groups, the  second 
assertion is  also immediate from Brauer's third main theorem.
 \end{proof}

\section{Klein four blocks, and proof of Theorem \ref{Kleinfour-thm}} 
\label{Kleinfour-section} 

We assume in this section that $p=2$ and that $k$ is a field of characteristic $2$.
Let $G$ be a finite
group, $B$ a block of $kG$ with a Klein four defect group $P$, and let $i\in B^P$
be a source idempotent. Assume that $k$ is a splitting field for $B$, its Brauer
correspondent, and their Brauer pairs.
By Erdmann \cite{Erdmann82}, $B$ is Morita equivalent
to either $kP$, $kA_4$, or $kA_5b_0$, where $b_0$ is the principal block idempotent
of $kA_5$.
The main result of \cite{CEKL}, which requires the classification of finite simple groups,
shows that in fact the source algebra $A=iBi$  of $B$ is isomorphic to either 
$kP$, $kA_4$, or $kA_5b_0$. 
The proof of Theorem \ref{Kleinfour-thm} proceeds by investigating these three
possible cases for the source algebras of $B$. Behind the scene is the fact there
is a stable equivalence between $B$ and its Brauer correspondent given by
truncated induction/restriction (see \cite[Section 12.2]{LiBookII}), so this is, 
just as in the cyclic block case, another special case of the scenario considered in 
Theorem \ref{simpleminded-Thm}.

\begin{proof}[Proof of Theorem \ref{Kleinfour-thm}]
If $A\cong kP$, then $A\tenkP k\cong$ $k$, which has trivially the endomorphism 
algebra $k$ as stated.
If $A\cong kA_4$, then $A\tenkP k\cong kC_3$, which is the direct sum of the
three pairwise non-isomorphic $1$-dimensional $A$-modules, and hence
the endomorphism algebra of $A\tenkP k$ is isomorphic to $k\times k\times k$.

\medskip
Suppose finally that  $A\cong kA_5b_0$. The algebra $kA_5b_0$ has three isomorphism 
classes of pairwise 
non-isomorphic simple modules $S_1$, $S_2$, $S_3$ of dimensions $1$, $2$, $2$,
respectively. Since $A_4\cong$ $P\rtimes C_3$ it follows
that the algebra $kA_4$ has three $1$-dimensional
modules $T_1$, $T_1$, $T_3$, with $T_1$ the trivial module, and we have
$kA_4\tenkP k \cong$ $kC_3\cong$ $T_1\oplus T_2\oplus T_3$.
Thus we have
$$b_0\Ind_{A_4}^{A_5}(k)= kA_5b_0\tenkP k \cong 
kA_5b_0 \ten_{kA_4} kC_3 = $$
$$=b_0\Ind^{A_5}_{A_4}(T_1) \oplus
b_0\Ind^{A_5}_{A_4}(T_2) \oplus b_0\Ind^{A_5}_{A_4}(T_3).$$
The induced modules $\Ind^{A_5}_{A_4}(T_i)$ are all $5$-dimensional; the truncation
by $b_0$ discards the defect zero simple modules, if any. 
As a very special case
of \cite[Theorem 3]{Erdmann82}, we may choose notation such that
 the restrictions to $kA_4$ of the simple
$kA_5b_0$-modules $S_1$, $S_2$, $S_3$ are isomorphic to $T_1$, $T_{2,3}$, $T_{3,2}$,
respectively, where $T_{i,j}$ denotes a uniserial module with composition factors
$T_i$, $T_j$, from top to bottom. Since the trivial $kA_5$-module is the Green
correspondent of the trivial $A_4$-module, it follows that
$\Ind^{A_5}_{A_4}(T_1)\cong$ $S_1\oplus Q$ for some necessarily $4$-dimensional
projective module $Q$ (corresponding to the irreducible character of degree $4$
of $A_5$). By Theorem \cite[Theorem 4]{Erdmann82}, all
projective indecomposable modules in the principal block of $A_5$ have dimension
strictly greater than $4$, so we get that 
$$b_0\Ind^{A_5}_{A_4}(T_1)\cong  S_1.$$
By Frobenius reciprocity, any nonzero $kA_5$-homomorphism 
$\Ind^{A_5}_{A_4}(T_2)\to S_i$ corresponds to a nonzero $kA_4$-homomorphism
$T_2\to \Res^{A_5}_{A_4}(S_i)$. The composition factors of the restrictions 
$\Res^{A_5}_{A_4}(S_i)$ imply that there is such a nonzero homomorphism only
when $i=3$, so the top of $\Ind^{A_5}_{A_4}(T_2)$ is the simple $kA_5$-module
$S_3$, or in other words, $\Ind^{A_5}_{A_4}(T_2)$ is a quotient of a projective
cover  of $S_3$. By \cite[Theorem 4]{Erdmann82}, a projective cover of
$S_3$  is uniserial of length five, 
with composition factors $S_3$, $S_1$, $S_2$, $S_1$, $S_3$. Comparing dimensions
shows that $\Ind^{A_5}_{A_4}(T_2)$ is the $5$-dimensional uniserial module $U_{3,1,2}$
of length $3$ with composition factors $S_3$, $S_1$, $S_2$. This module belongs to
the principal block, so is equal to $b_0\Ind^{A_5}_{A_4}(T_2)$.
A similar argument, exchanging the roles of $T_2$ and $T_3$ shows that
$\Ind^{A_5}_{A_4}(T_3)=$ $b_0\Ind^{A_5}_{A_4}(T_3)$ is a uniserial
module $U_{2,1,3}$ with composition factors $S_2$, $S_1$, $S_3$.
It follows that
$$\End_A(A\tenkP k) \cong \End_{kA_5}(S_1 \oplus U_{3,1,2} \oplus U_{2,3,1}).$$
Since $S_1$ appears neither in the top nor in the socle of the two modules
$U_{3,1,2}$, $U_{2,1,3}$, this summand gives rise to a direct algebra factor
isomorphic to $k$, and by considering the composition series of $U_{3,1,2}$
and $U_{2,1,3}$ one sees easily that $\End_{kA_5}(U_{3,1,2}\oplus U_{2,1,3})$ is
a $4$-dimensional Nakayama algebra with two isomorphism classes of
simple modules and projective modules of length $2$.
Theorem \ref{Kleinfour-thm} follows.
\end{proof}

\begin{Remark} \label{Kleinfour-remark}
Without the classification of finite simple groups the source algebras of blocks 
with Klein four defect groups are determined in \cite{LiKleinfour} only up to
a power of the Heller operator (see \cite[Theorem 1.1]{LiKleinfour} for a precise
statement). This 
is closely related to \cite[Theorem 3]{Erdmann82}, where the Green correspondents
of simple modules for the normaliser of a defect group are determined up to 
a power of the Heller operator. If we replace the above Green correspondents
by a Heller translate, then it turns out that 
the corresponding endomorphism algebra is no
longer self-ijective in general, as we illustrate below. Thus the self-injectivity
of $\End_B(Bi\tenkP k)$ in Theorem \ref{Kleinfour-thm}  seems to be related
to Puig's finiteness conjecture on source algebras with a fixed  defect group.
Puig's finiteness conjecture for blocks with a Klein four group
 is at present only accessible using  the classification of finite simple groups
(see \cite{CEKL}). 

In order to describe this in more detail, we use 
the fact that the projective indecomposable modules for any of the
three algebras $kP$, $kA_4$, $kA_5b_0$ have hearts which are the direct
sum of at most two nonzero  uniserial modules, and so their structures are
determined by their Loewy layers. We also note that since these algebras 
are symmetric, hence self-injective, 
any endomorphism of the unique  maximal submodule $\rad(X)$  of a
projective indecomposable module $X$ is the restriction to $\rad(X)$ of 
an endomorphism of $X$. 

With the notation above, the Loewy layers of the regular $kP$-module $kP$ are 
$$\begin{matrix} & k & \\ k & & k \\ & k & \end{matrix}  $$
and hence the Loewy layers of its unique maximal submodule $\Omega_P(k)$
are
$$\begin{matrix}  k & & k \\ & k & \end{matrix}  $$
Thus the endomorphism algebra of $\Omega_P(k)$ is
a $3$-dimensional commutative algebra with radical square zero, so not
self-injective. The quiver of this algebra is
$$\xymatrix{\bullet \ar@(ld,lu) \ar@(rd,ru) }$$

The three simple $kA_4$-modules $T_1$, $T_2$, $T_3$ have projective
covers with Loewy layers
$$
\begin{matrix} & T_1 & \\ T_2 & & T_3 \\ & T_1 & \end{matrix} \hskip2cm
\begin{matrix} & T_2 & \\ T_1 & & T_3 \\ & T_2 & \end{matrix} \hskip2cm
\begin{matrix} & T_3 & \\ T_2 & & T_4 \\ & T_3 & \end{matrix} 
$$
The Heller translates $\Omega_{A_4}(T_i)$ of the simple modules $T_i$
are the maximal submodules of
their projective covers, hence have Loewy layers
$$
\begin{matrix} T_2 & & T_3 \\ & T_1 & \end{matrix} \hskip2cm
\begin{matrix}  T_1 & & T_3 \\ & T_2 & \end{matrix} \hskip2cm
\begin{matrix}  T_2 & & T_4 \\ & T_3 & \end{matrix} 
$$
The composition factors in these modules show that there is a nonzero
homomorphism between any two different of these Heller translates, and any two of
these homomorphisms compose to zero.
A straightforward calculation shows that
the algebra 
$$\End_{kA_5}(\Omega_{A_4}(T_1)\oplus \Omega_{A_4}(T_2)\oplus 
\Omega_{A_4}(T_1))$$ 
is a radical square zero algebra with quiver
$$
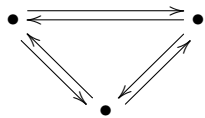
\xymatrix{\bullet \ar@<+.8ex>[rr] \ar@<-.8ex>[dr] & & \bullet \ar[dl] \ar[ll] \\
 & \bullet \ar[ul] \ar@<-.8ex>[ur] }$$
 which is a $9$-dimensional algebra with three projective indecomposable
 modules whose socles are the direct sum of two simple modules, so also
 not self-injective.
 
 The three $kA_5b_0$-modules $S_1$, $U_{3,1,2}$, $U_{2,1,3}$
 have, by \cite[Theorem 4]{Erdmann82},  projective covers with Loewy layers
 $$
 \begin{matrix} & S_1 & \\ S_2 & & S_3 \\ S_1 & & S_1 \\ S_3 & &  S_2 \\ & S_1 
 \end{matrix} \hskip2cm
 \begin{matrix} S_2 \\ S_1 \\ S_3 \\ S_1 \\ S_2 \end{matrix} \hskip2cm
 \begin{matrix} S_3 \\ S_1 \\ S_2 \\ S_1 \\ S_3 \end{matrix}
 $$
 Thus the Heller translates of $S_1$, $U_{3,1,2}$ and $U_{2,1,3}$ have
 Loewy layers
 $$
 \begin{matrix} S_2 & & S_3 \\ S_1 & & S_1 \\ S_3 & &  S_2 \\ & S_1 
 \end{matrix} \hskip2cm
 \begin{matrix} S_1 \\ S_2  \end{matrix} \hskip2cm
 \begin{matrix} S_1 \\ S_3  \end{matrix}
 $$
 The first of these three modules has two endomorphisms with images
 of length two and composition factors either $S_3$, $S_1$ or $S_2$, $S_1$.
 There are nonzero homomorphisms to this module from the other two modules
 but not from the first module to the other two modules. 
 It follows that the resulting endomorphism algebra
 $$\End_{kA_5b_0}(\Omega_{A_5}(S_1) \oplus \Omega_{A_5}(U_{3,1,2})
 \oplus \Omega_{A_5}(U_{2,1,3}))$$
 is a $7$-dimensional radical square zero algebra with quiver
 $$
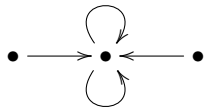
\xymatrix{ & & \\ 
 \bullet \ar[r] & \bullet \ar@(ul, ur) \ar@(dl,dr) & \ar[l] \bullet \\
 & & }$$
This algebra has two projective simple modules which are not injective, so this
algebra is yet again not self-injective. It would be interesting to calculate this for all
Heller translates $\Omega^n$ of these modules, with $n$ a nonzero integer.
If it were the case that all endomorphism algebras as above with $\Omega^n$
instead of $\Omega$ for some nonzero integer $n$ are not self-injective, then the 
self-injectivity of $\End_B(Bi\tenkP k)$ would imply  Puig's finiteness conjecture 
for Klein four defect blocks.
\end{Remark}

\section{Nilpotent blocks, and proof of Theorem \ref{nilpotent-thm} }
\label{nilpotent-Section}   

Let $k$ be  a field of prime characteristic $p$. Let $G$ be a finite group and $B$ a 
block of $kG$ with a defect group $P$.
We assume that $k$ is large enough for $B$ and its Brauer pairs.
Following \cite{BrPunil}, the block $B$ is {\em nilpotent} if any fusion system
of $B$ on $P$ is equal to the trivial fusion system $\CF_P(P)$ of $P$ on itself.
See \cite[Section 8.11]{LiBookII} for an expository account on nilpotent blocks.
Puig has determined in \cite{Punil} the structure of the source algebras of
nilpotent blocks. (The field $k$ in \cite{Punil} is assumed to be algebraically closed,
but the above weaker hypotheses on $k$ are sufficient - see 
\cite[Remark 1.14]{Punil}; see also \cite[Theorem 1.22]{FanPuig} for more precise 
results.)  

Assume that the block $B$ of $kG$ is nilpotent. Let $i\in B^P$ be a source idempotent.
Then there is an indecomposable endopermutation $kP$-module with vertex $P$
such that, setting $S=\End_k(V)$, we have an isomorphism of interior $P$-algebras
$$iBi \cong S \tenkP kP.$$
The interior $P$-algebra structure is here given by the diagonal map sending $u\in P$
so $\sigma(u)\ten u$, where $\sigma(u)$ is the linear endomorphism of $V$ given
by the action of $u$ on $V$. The source algebra, and hence $B$ itself, is Morita
equivalent to $kP$. Thus the  block $B$ has a unique isomorphism class of
simple modules. Moreover, the $kP$-module $V$ is a source of any simple 
$B$-module.  The isomorphism class of $V$ is unique up to conjugation by $N_G(P)$.

\begin{Lemma} \label{nilpotent-Lemma}
With the notation above, the map sending $s\in S$ to $(s\ten 1_P)\ten 1_k$
induces an isomorphism of $S\tenk kP$-modules
$$S \cong (S\tenk kP) \tenkP k.$$
Here $S$ is regarded as a left $S\tenk kP$-module with $s\ten u$ acting on $t$
by $st\sigma(u^{-1})$, where $s$, $t\in S$ and $u\in P$, and the right
$kP$-module structure on $S\tenk kP$ is given by right multiplication with
$\sigma(u)\ten u$, where $u\in P$. In particular, we have
$$\End_{S\tenk kP}((S\tenk kP) \tenkP k) \cong \End_{S\tenk kP}(S)\cong 
(S^\op)^P \cong \End_{kP}(V)^\op.$$
\end{Lemma}

\begin{proof}
The map sending $s\in S$ to $(s\ten 1_P)\ten 1_k$ is clearly surjective, and since
both sides have the same dimension this map is a linear isomorphism. It is
obviously a left $S$-module homomorphism as well. We need to check that
this is also a homomorphism for the action of $kP$. On the left side, $u\in P$ 
acts as right multiplication with $\sigma(u)$. On the right side, $u\in P$ acts
as multiplication with $1_S\ten u$.  We need to check that these actions commute
with the map $s\to$ $(s\ten 1_P)\ten 1_k$. Now this map sends $s\sigma(u^{-1})$ to
$(s\sigma(u^{-1}) \ten 1_P)\ten 1_k=$ $(s\sigma(u^{-1}) \ten 1_P)\ten u\cdot 1_k=$
$s\sigma(u^{-1})\sigma(u)\ten u)\ten 1_k=$ $(s\ten u)\ten 1_k$, whence the first
isomorphism. Thus we have $\End_{S\tenk kP}((S\tenk kP) \tenkP k) \cong $
$\End_{S\tenk kP}(S)$. Every endomorphism of $S$ as a left $S$-module is
induced by right multiplication with an element in $S$, and for this to be also
a $kP$-endomorphism we need that this element is in $S^P$. The result follows.
\end{proof} 

\begin{Remark}
For future reference we note that the arguments in the proof of Lemma
\ref{nilpotent-Lemma} hold in much greater generality: if $H$ is a finite group
and $B$ an interior $H$-algebra over some Noetherian local principal ideal domain
$R$ such that $B$ is free of finite rank as an $R$-module, 
then the map sending $b\in B$ to $(b\ten 1)\ten 1\in (B\tenR RH)\ten_{RH} R$
is an isomorphism of $B\tenR RH$-modules $B\cong$ $(B\tenR RH)\ten_{RH} R$,
and hence $\End_{B\tenR RH}((B\tenR RH)\ten_{RH} R)\cong (B^\op)^H$.
\end{Remark}

\begin{proof}[Proof of Theorem \ref{nilpotent-thm}]
We use the notation above. As mentioned in \ref{EndBi-isom}, by
the standard Morita equivalence between
$B$ and its source algebra $A=iBi$, we have an isomorphism
$\End_{kG}(Bi\tenkP k)\cong$ $\End_A(A\tenkP k)$. It follows from
Lemma \ref{nilpotent-Lemma} that this algebra is isomorphic to
$\End_{kP}(V)^\op$, where the notation is as above. 
If this  endomorphism  algebra is self-injective, then by a result of Green 
(as presented in \cite[Theorem 2.13]{Geck01}) has a simple top and socle, 
whence the Theorem.
\end{proof}

\begin{Remark}
Except for the case where $P$ is cyclic, the property in Theorem \ref{nilpotent-thm}
 that an indecomposable
endopermutation $kP$-module $V$ with vertex $P$ has a simple top and socle is
very restrictive. In the case where $p=2$ and $P$ a Klein four group, this property
forces  $V$ to be  trivial (cf. the Remark \ref{Kleinfour-remark} above).  More generally,
for any prime $p$, if 
$P$ is a finite abelian $p$-group, then the Heller translates of $k$ are, up to
isomorphism, all endotrivial $kP$-modules, and if in addition $P$ is non-cyclic, then
the trivial $kP$-module is the only endotrivial $kP$-module with a simple top and 
socle. 
\end{Remark}

\section{Symmetric groups, and proof of Theorem \ref{cyclic-Sp}} 
\label{Example-section} 

Let $k$ be  a field of prime characteristic $p$.  The   normaliser in $S_p$ of a Sylow 
$p$-subgroup  $S\cong C_p$   is isomorphic to  $S\rtimes E$, where $E$ is  a cyclic 
group of  order $p-1$. Thus  
$\Ind_{S}^{S_p} (k) = $ $\oplus_X  \Ind_{S\rtimes E}^{S_p} (X)$, 
where $X$ runs  over a set of  representatives of  isomorphism classes of 
simple  $k(S\rtimes E)$-modules.  Further, for any simple $kS\rtimes E$-module 
$X$,   $ \Ind_{S \rtimes  E}^{S_p} (X) $  is the direct sum  of the Green 
correspondent of  $X$ and    a projective  $kS_p$-module.   Thus, the non-projective 
part of $\Ind_S^{S_p} (k)$  is the sum of   the  Green correspondents of  the  simple 
$ kS\rtimes E$-modules each occurring with multiplicity $1$. Since the principal 
block of  $kS_p$ is the only block with non-trivial  defect,  and $S$ is a  
vertex of all simple $S\rtimes  E$-modules, their  Green correspondents   
all lie in the principal block.

Let $B$  be  the  principal block of $kS_p $. The Brauer tree of  $B$  is a 
straight line with $p-1$ edges, labelled by simple $B$-modules $T_1$, $T_2$,.., 
$T_{p-1}$ and  vertices, none exceptional, labelled by  irreducible characters 
$\chi_s$, $1\leq s\leq p$, 
$$\xymatrix{\bullet^{\chi_1} \ar@{-}[r]_{T_1} & 
\bullet^{\chi_2} \ar@{-}[r]_{T_2} & \bullet^{\chi_3} \ar@{-}[r]_{T_3} & \cdots \ar@{-}[r] &
\bullet^{\chi_{p-1}} \ar@{-}[r]_{T_{p-1}} & \bullet^{\chi_p}}$$
such that $\chi_s+\chi_{s+1}$ is the character of a projective cover of $T_s$,
for $1\leq s\leq p-1$. Note that $p$ divides therefore $\chi_s(1)+\chi_{s+1}(1)$. 
Thus the characters $\chi_s$ with $s$ odd have degrees congruent to $1$ 
modulo $p$, while the characters $\chi_s$ with $s$ even have degrees congruent 
to $-1$ modulo $p$. Note that by  Proposition \ref{proj-free-Frob-Prop} the 
choice of field  $k$  is irrelevant.

Since the simple $k(S\rtimes E)$-modules are $1$-dimensional, their Green
correspondents have dimensions congruent to $1$ modulo $p$. 
The Green correspondent of the trivial $k(S\rtimes E)$-module is the
trivial $kS_p$-module $T_1$, and the Green correspondent of the $1$-dimensional
module which sends a generator of $S$ to $-1$ is the sign representation $T_{p-1}$.
The Green correspondents of the remaining simple $k(  S\rtimes E)  $-modules
are uniserial of length $2$, with composition factors $T_{2m}$, $T_{2m+1}$ and
$T_{2m+1}$, $T_{2m}$, for $1\leq m\leq \frac{p-3}{2}$, and both of these lift 
canonically to modules with character $\chi_{2m+1}$. The characters $\chi_{2m}$ 
do not appear in lifts of the Green correspondents of the simple 
$k(S\rtimes E)$-modules  and the  contribution of any irreducible character of $B$ 
to the non-projective part of   $\Ind_{S}^{S_p}(k)$ is  at most $2$.

\begin{Lemma}  \label{lem:proj non-simple}   
Let   $S$  be a Sylow $p$-subgroup of  $S_p$ and  $B$  the principal block of 
$kS_p$. If $p \geq 7$, then $B \tenkS k $ contains  a  projective  
non-simple indecomposable   summand.  
\end{Lemma}  

\begin{proof}  
The  characters in the  principal block  of $kS_p$ are those labelled by hook 
partitions of $p$, and   by an easy application of Murnaghan-Nakayama rule  
it follows that if $\chi  \in \Irr (S_p)$  is  labelled by   the hook partition   
$(p-i,  1^i )$, $ 1\leq  i \leq p-1) $, then 
$$\langle  \chi,  \Ind_{C_p} ^{S_p} k \rangle =  
\langle  \Res^{S_p}_{C_p}  \chi,  k  \rangle  = 
\frac{1}{p} \big( \binom{p-1}{i}  +(-1)^i \big) .$$
In particular, if $ p\geq 7 $ and $i=2$, then  
$\langle  \chi,  \Ind_{C_p} ^{S_p} k \rangle= $  $\frac{p-1}{2} \geq 3 $. 
 The result follows  since  as  observed above the  contribution of   $\chi $ to   
 the non-projective part of  $\Ind^{S_p}_{C_p}(k)$ is  at most $2$.
 \end{proof} 

\begin{Lemma} \label{lem:cliff} 
Suppose that $H$ is a normal subgroup of  a finite group $G$.
Let $X$ be an indecomposable  $kG$-module  and  $Y$ 
a projective  indecomposable   summand of $\Res^{G}_{H}  X$. Suppose that   
the  $kH$-modules    $Y $ and  $xY$   are non-isomorphic  for any 
$ x \in G\setminus  H$. Then   $ X \cong \Ind_H^G (Y)$  and in particular, 
$X$ is projective.   Moreover, $X$ is simple if and only if $Y$ is  simple.
\end{Lemma}

\begin{proof} 
Let  $I$  be a set of coset representatives of $H$ in $G$, and let $X_0 =$
$\sum_{x\in I }  xY $, the smallest $kG$-submodule of $X$ containing $Y$.  Since  
 $Y$ is   projective  and indecomposable and $kG$ is symmetric,  $Y$ has a simple 
 socle  and  $Y$  is the  injective envelope of its  socle.  So,  if  $Y$ and $xY $ are 
 non-isomorphic, then   $\soc(Y) \cap \soc(xY)=0 $  and consequently, $Y \cap xY =0 $. 
 Thus,   $X_0$ is  a  direct sum of the $xY$'s,  $x\in I$.      By  Frobenius reciprocity, 
 the inclusion of $Y$ into $ \Res^{G}_{H}  (X_0)  $  gives  rise  to  a  $kG$-homomorphism, 
 say $\varphi$,    from $\Ind_H^G (Y) $  to $X_0$  satisfying    $ 1 \otimes y =  y  $ for 
 all   $ y\in Y$.   Since $X_0 $ is the smallest $kG$-submodule   of $X$ containing $Y$, 
 the homomorphism   $\varphi $  is clearly surjective.   But $X_0 $ has the same 
 $k$-dimension as  $Y$. Thus $\varphi $ is an isomorphism  and in particular $X_0$ 
 is  a projective and hence an  injective  $kG$-module. Thus $ X_0$ is a direct 
 summand of $X$. Since $X$ is indecomposable,  $X_0=X$ and 
and  the  first  two  assertions follow.  
The  last assertion is  immediate from the first one.
\end{proof}

\begin{Lemma} \label{lem:wreath}   
Let $H$  be  a finite group, $Q$ a Sylow $p$-subgroup  of  $H$,  $G= H\wr C_p$ and 
$P$  a Sylow $p$-subgroup of $G$.  
Suppose  that  $\Ind_Q^H (k) $  has  a projective   and 
non-simple indecomposable  summand  and that   $kH$ also has a  simple projective 
module.  Then $\Ind_{P}^G(k) $     has  a  projective   and non-simple indecomposable 
summand.
\end{Lemma}

\begin{proof}  
Let $L = H \times \cdots\times H \cong H^p$ be the  base subgroup of $G$. Let  $Q$ 
be  a  Sylow $p$-subgroup  of $H$, $ R=  Q \times\cdots   \times  Q $,  a Sylow 
$p$-subgroup of $L$. Without loss  we may assume that  $ P$  contains  $R$.  
Then $ G =  LP$  and the Mackey formula gives
$$ \Res^G_L \Ind_P^G  (k)  = \Ind_R^L (k).$$
Let $M$ be  a  projective indecomposable  and non-simple summand  of     
$\Ind_{Q}^H(k) $ and let $N$ be a projective simple  $kH$-module. Then  
$N$ is  also a  summand of  $\Ind_{Q}^H(k) $.   Consider the $kL$-module
$$ Y :=  M  \otimes  N \otimes  \cdots \otimes N.   $$  Then  $ Y$ is an indecomposable 
projective and non-simple  summand of  $\Ind_R^L (k) $.   By  the above displayed 
equation, there exists an indecomposable factor, say $X$ of $ \Ind_P^G  (k) $ such 
that  $Y$ is a summand of $\Res^G_L (X) $.

Since    elements of $G\setminus L$ cyclically permute the   $p$-factors  of $L$, 
it follows that $Y $  and $xY$ are non-isomorphic  for any $x \in G\setminus L $.  
Lemma~\ref{lem:cliff}  applied with $L$ in place of $H$  gives that 
$X \cong \Ind_L^G (Y) $ is  projective  and non-simple.
\end{proof}  

We need the following consequence of a result of Kn\"orr 
\cite[2.13 Corollary]{Knoerr77}.

\begin{Lemma} \label{primetop}  
 Let $G$ be  a finite group, $P$ a Sylow $p$-subgroup of $G$ and $H$  a  subgroup 
 of $G$ containing $P$.   If  $\Ind_P^H(k)$  has  a projective  and non-simple  
 indecomposable  summand, then so does $\Ind_P^G (k) $.
\end{Lemma}

\begin{proof} 
Let $Y$ be a projective non-simple indecomposable summand of $\Ind^H_P(k)$.
Then $\Ind^G_H(Y)$ is a projective summand of $\Ind^G_P(k)$, and by
\cite[2.13 Corollary]{Knoerr77} it is not semisimple.
\end{proof}

\begin{Lemma} \label{ngreaterp}  Suppose that $n \geq  p \geq 7 $ and let  $P$ be a 
Sylow $p$-subgroup of $S_n $. Then $\Ind_P^{S_n} (k)$ contains an 
indecomposable  projective  and non-simple summand.    
\end{Lemma} 

\begin{proof}
We  use induction on $n$.   If $n=p $,  we are done by  Lemma \ref{lem:proj non-simple}.    
Now suppose that $ n >p $ and the  property   holds  for all $ m $, $ p\leq m <n $.    
If $P \leq S_m $ for some $m < n $, then the property  holds for $S_n$  by 
Lemma  \ref{primetop}.
Suppose next that   $ P \leq  S_{n_1} \times S_{n_2} $ with $n=n_1 +n_2  $  
and $n_1 \geq p,  n_2 \geq p $.  Since the property  clearly passes to direct 
products, again we are done by Lemma \ref{primetop}.
Thus, we may assume that  $n=p^a $  for some $a \geq 2 $.   Then there exists  a 
subgroup  $G$ of $S_n $ containing $P$   and  isomorphic to  $ S_{p^{a-1}} \wr C_p $.  
By   a result of    Granville and Ono   \cite[Theorem 1]{GO},   $kS_{p^{a-1}} $    has a 
projective  simple summand.   Hence  by Lemma \ref{lem:wreath}    we have that  
$\Ind_P^G (k) $  contains  an indecomposable, projective, non-simple summand. 
By Lemma  \ref{primetop}, so does $ \Ind^{S_n}_P(k) $.
\end{proof} 

\begin{proof}[{Proof of Theorem \ref{cyclic-Sp}}]   
(a)   If $ n=p $  and  $B$ is the principal block, the  statement is  trivially true  for 
$p=2 $ and for odd $p$ it  follows from the proofs  of  Theorem \ref{cyclic-thm}   
and Theorem \ref{Brauertree-thm-2}  and from  the description of  the Brauer 
tree of  $B$ given  in the beginning of this section.   Here we note    that  an
indecomposable  
Nakayama algebra is symmetric if and only if it  is  a star Brauer tree algebra 
(with exceptional vertex, if any, in the center; see  \cite[Sections 11.6, 11.7]{LiBookII}). 
 By a result of Puig \cite{PuScopes},   any block of    $kS_n $ with non-trivial cyclic 
 defect groups  is  source algebra 
 equivalent   to the  principal  block of $kS_p$,   hence   the general case of (a) follows 
 from the special   case  and \ref{EndBi-isom}.  Note that the field $k$ is assumed 
 algebraically closed in  
 \cite{PuScopes},  but it can be checked that this assumption is not needed.

(b)    The first   assertion  follows from
Proposition \ref{proj-free-Frob-Prop}.   The remaining assertions are  trivially true  
if   $p=2$    or $3$  and    they  follow   from     \cite[5.6.4,  5.6.4]{NaehrigPhD}  and 
Corollary \ref{noprojective-cor}   if $p=5 $  or $p =7 $  and  from  Lemma 
\ref{lem:proj non-simple}   and  Corollary \ref{noprojective-cor}  if $p>7 $.

(c)  This is immediate  from  Lemma \ref{ngreaterp}  and  Corollary 
\ref{noprojective-cor}.
\end{proof}

\begin{Example} \label{S7-example}
Suppose that $k$ is a field of characteristic $7$. The symmetric group algebra
$kS_7$ has a principal block $B$ with a cyclic defect group $S=C_7$ of order $7$, and
all other blocks have defect zero.  Although $\End_B(B\tenkS k)$ is not
self-injective, it is a quotient of a  Brauer tree algebra after
removing trivial factors - this curious
fact is the reason  for including the detailed calculations below.

The block  $B$  has six pairwise non-isomorphic
simple modules $\{T_i\}_{1\leq i\leq 6}$, of dimensions $1$, $5$, $10$, $10$, $5$, $1$, 
and seven irreducible characters $\{\chi_i\}_{1\leq i\leq 7}$ of degrees
$1$, $6$, $15$, $20$, $15$, $6$, $1$, respectively. We choose notation such that
$T_1$ is the trivial $kS_7$-module and $T_6$ is the sign representation. 
The Brauer tree of $B$ is a straight line with six edges and no exceptional multiplicity
$$\xymatrix{\bullet^{\chi_1} \ar@{-}[r]_{T_1} & 
\bullet^{\chi_2} \ar@{-}[r]_{T_2} & \bullet^{\chi_3} \ar@{-}[r]_{T_3} & 
\bullet^{\chi_4} \ar@{-}[r]_{T_4} & \bullet^{\chi_5} \ar@{-}[r]_{T_5}  & 
\bullet^{\chi_6} \ar@{-}[r]_{T_6} & \bullet^{\chi_7} }$$
The projective indecomposable $B$-modules $P_i$ corresponding to the simple modules
$T_i$, $1\leq i\leq 6$, have Loewy layers
$$\begin{matrix} T_1 \\ T_2 \\ T_1 \end{matrix}\hskip12mm
\begin{matrix} & T_2 & \\ T_1 & & T_3 \\ & T_2 & \end{matrix} \hskip12mm
\begin{matrix} & T_3 & \\ T_2 & & T_4 \\ & T_3 & \end{matrix} \hskip12mm
\begin{matrix} & T_4 & \\ T_3 & & T_5 \\ & T_4 & \end{matrix} \hskip12mm
\begin{matrix} & T_5 & \\ T_4 & & T_6 \\ & T_5 & \end{matrix} \hskip12mm
\begin{matrix} T_6 \\ T_5 \\ T_6 \end{matrix}
$$
and the dimensions of the $\{P_i\}_{1\leq i\leq 6}$ are $7$, $21$, $35$, $35$, $21$, $7$,
respectively,from which we get as in the previous example  that 
$\dim_k(B)=$ $2\cdot(7+5\cdot 21 + 10\cdot 35)=$ $924$.
The character of the trivial source module lifting $B\tenkP k$ is
$\chi_1+3\chi_3+2\chi_4+3\chi_5+\chi_7$, 
which we can rewite as 
$$\chi_1 + \chi_3 + \chi_5 + \chi_7 + (\chi_3+\chi_4) + (\chi_4+\chi_5)$$
The characters  $\chi_1$, $\chi_3$, $\chi_5$, $\chi_7$ are the characters of
the canonical lifts of the uniserial $B$-modules corresponding to the simple 
$k(P\rtimes E)$-modules, and the characters $\chi_3+\chi_4$ and $\chi_4+\chi_5$
are the characters of the projective indecomposable modules lifting $P_3$ and $P_4$.
Thus $B\tenkS k$ is the direct sum of the modules
$$\begin{matrix} T_1\end{matrix} \hskip13mm
\begin{matrix} T_2\\T3\end{matrix} \hskip13mm
\begin{matrix} T_3\\T_2\end{matrix} \hskip13mm
\begin{matrix} P_3\end{matrix} \hskip13mm
\begin{matrix} P_4\end{matrix} \hskip13mm
\begin{matrix} T_4\\T_5\end{matrix} \hskip13mm
\begin{matrix} T_5\\T_4\end{matrix} \hskip13mm
\begin{matrix} T_6\end{matrix} $$

The algebra $\End_{B}(B\tenkS k)$  is isomorphic to $k \times \End_{B}(M') \times k$, 
where
$$M'=
\begin{matrix} T_2\\T3\end{matrix} \oplus
\begin{matrix} T_3\\T_2\end{matrix} \oplus
\begin{matrix} P_3\end{matrix} \oplus
\begin{matrix} P_4\end{matrix} \oplus
\begin{matrix} T_4\\T_5\end{matrix} \oplus
\begin{matrix} T_5\\T_4\end{matrix} .$$
It is easy to write down all homomorphisms between these summands, and one 
obtains that the quiver and relations of $\End_{B}(M')$ are as follows (where we
label the vertices by the indecomposable summands of $M'$, denoting by $T_{i,j}$ the
length $2$ uniserial modules with composition factors $T_i$, $T_j$ from top to 
bottom):

$$
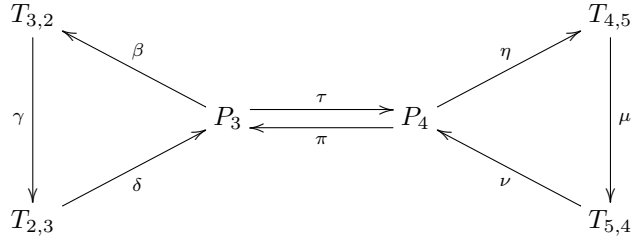
\xymatrix{ T_{3,2}  \ar[dd]_{\gamma} & & & & & & 
T_{4,5}  \ar[dd]^\mu \\
          & & P_3 \ar@<+.8ex>[rr]^{\tau} \ar[llu]_{\beta}& &
                P_4 \ar@<+.8ex>[ll]^{\pi} \ar[rru]^{\eta} & &    \\
        T_{2,3} \ar[rru]_{\delta} & & & & & & 
          T_{5,4}  \ar[llu]^{\nu}   }
$$
We compose arrows left to right.
The arrow $\tau$ corresponds to a homomorphism with image $T_{3,4}$ and kernel
$T_{2,3}$; the arrow $\pi$ corresponds to a homomorphism with image $T_{4,3}$ and
kernel $T_{5,4}$. Thus $\tau\pi$ corresponds to an endomorphism of $P_3$ with
image $\soc(P_3)$, and $\pi\tau$ corresponds to an endomorphism of $P_4$ with
image $\soc(P_4)$. The arrow $\beta$ corresponds to a surjective map $P_3\to T_{3,2}$,
the arrow $\gamma$ corresponds to a homomorphism $T_{3,2}\to T_{2,3}$ with
image $T_3$, the arrow $\delta$ corresponds to an injective map $T_{2,3}\to P_3$.
The arrows $\eta$, $\mu$, $\tau$ correspond to the analogous homomorphisms
at the other end of the quiver. The relations are easily read off the composition
series of the involved modules: after possibly adjusting the homomorphisms
corresponding to  $\gamma$ and $\mu$ by a nonzero scalar multiple,  we have
$$\beta\gamma\delta = \tau\pi, \ \ \ \eta\mu\nu = \pi\tau$$
which imply that $\pi\beta=0$ and $\tau\eta=0$. Further relations are
$$\gamma\delta\beta = \delta\beta\gamma = \mu\nu\eta=\nu\eta\mu=0.$$
The above quiver is obtained by attaching two $3$-cycles and one $2$-cycle 
to each other in such a way that to each vertex there are at most two cycles attached.
Thus this is the quiver of a Brauer tree algebra, namely  of the Brauer tree
$$
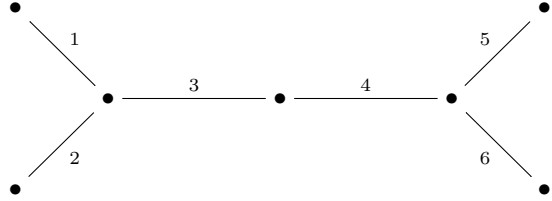
\xymatrix{ \bullet \ar@{-}[dr]^1 & & & & & &\bullet \\
 & \bullet \ar@{-}[rr]^3 & & \bullet \ar@{-}[rr]^4 & & \bullet \ar@{-}[ur]^5 \ar@{-}[dr]_6 &  & \\
 \bullet \ar@{-}[ur]_2& & & & & & \bullet  
 }$$
 This Brauer tree algebra has six isomorphism classes of simple modules, and (denoting
 these by the integers used to label the edges of the tree) their projective covers
 have Loewy series
 $$
 \begin{matrix} 1 \\ 2 \\ 3 \\ 1 \end{matrix} \hskip1cm
  \begin{matrix} 2 \\ 3 \\ 1 \\ 2 \end{matrix} \hskip1cm
  \begin{matrix} & 3 & \\ 1 & & 4 \\ 2 & & \\ & 3 & \end{matrix} \hskip1cm
   \begin{matrix} & 4 & \\ 3 & & 5 \\  & & 6 \\ & 4 & \end{matrix} \hskip1cm
   \begin{matrix} 5 \\ 4 \\ 6 \\ 5 \end{matrix} \hskip1cm
    \begin{matrix} 6 \\ 5 \\ 4 \\ 6 \end{matrix} 
 $$
 By checking the relations in this Brauer tree algebra one sees that the algebra
 $\End_{B}(M')$ is a quotient of this Brauer tree algebra by the $4$-dimensional
 ideal consisting of the simple modules in the socle labelled $1$, $2$, $5$, $6$.
 This quotient is no longer  self-injective, as it does no longer give rise to 
 a bijection between the isomorphism classes of top and bottom composition
 factors. The structure of the projective indecomposable $\End_{B}(M')$-modules
 is thus as follows:
 $$
 \begin{matrix} 1 \\ 2 \\ 3 \end{matrix} \hskip1cm
  \begin{matrix} 2 \\ 3 \\ 1  \end{matrix} \hskip1cm
  \begin{matrix} & 3 & \\ 1 & & 4 \\ 2 & & \\ & 3 & \end{matrix} \hskip1cm
   \begin{matrix} & 4 & \\ 3 & & 5 \\  & & 6 \\ & 4 & \end{matrix} \hskip1cm
   \begin{matrix} 5 \\ 4 \\ 6  \end{matrix} \hskip1cm
    \begin{matrix} 6 \\ 5 \\ 4  \end{matrix} 
 $$
 
\end{Example}

\section{Further remarks}
 Let $\CO$ be a complete discrete valuation ring with
residue field $k$ of prime characteristic $p$ and field of fractions $K$ of
characteristic zero.   Let $G$ be a finite group and let $S$ be a Sylow $p$-subgroup 
of $G$.

\begin{Remark} \label{O-Remark}
Let $B$ be a block of $\OG$ with defect group $P$.
Finitely generated  trivial source $kG$-modules lift uniquely, up to isomorphism, 
to trivial source  $\OG$-modules (see e.g. \cite[Theorem 5.10.2]{LiBookI}),
and hence finitely generated $p$-permutation $kG$-modules lift uniquely,
up to isomorphism, to $p$-permutation $\OG$-modules.
If $i$ is a source idempotent in $B^P$, then the image of $i$ in $kG$ is a source
idempotent of the corresponding block $\bar B\cong$ $k\tenO B$ of $kG$.
More generally, the canonical map $\OG\to$ $kG$ induces a bijection between
(local) pointed groups on $\OG$ and on $kG$.
The Sylow permutation module, its block components, the defect
permutation module and source permutation module of $B$ are all defined
with $k$ replaced by $\CO$, and some character of $B$ appears in 
the character of the Sylow permutation module $\Ind^G_S(\CO)$ over $\CO$,
and for the same reason, weight modules and their Green correspondents lift 
over $\CO$ as well. Thus the Theorems \ref{Bi-summand}, \ref{Bi-weight-thm},
\ref{Bi-no-projective-summand-thm} (i), 
 and \ref{princ-block-thm}
hold with $k$ replaced by $\CO$.  The first isomorphism in Theorem
\ref{nilpotent-thm} has an analogue over $\CO$.
The following results in Section \ref{backgroundsection}
lift  over $\CO$: 
the Lemmas \ref{ij-primitive-Lemma}, \ref{perf-Lemma}, 
\ref{bimod-Lemma}, \ref{perfect-vertices-Lemma},
 \ref{bimod-idempotent-Lemma},
Proposition \ref{iU-vertices}, Corollary \ref{iU-permutation-summands},
and Proposition \ref{BiMorita}.
All results in Section \ref{Sylow-section} have obvious analogues over $\CO$, except
the results involving simple modules  in Lemma \ref{no-proj-summands-Lemma} (i), 
Proposition \ref{Bi-top-socle},  and Proposition \ref{normal-Prop}. 
Since weight modules lift over $\CO$, so do the statements of 
Theorem \ref{Bi-weights}, Lemma \ref{Bi-summands-Lemma}, 
Corollary \ref{ellC-Cor}, Theorem \ref{Bi-princ-block},
and Lemma \ref{non-local-points-2}.
Proposition \ref{splendid-stable-source-Prop},
Theorem \ref{splendid-stable-source}, Lemma \ref{splendid-source-Lemma},
Corollary \ref{splendid-source-Cor}, 
and Proposition  \ref{splendid-stable-Prop} hold over $\CO$.
The statements (i), (ii), and (v) in Proposition \ref{splendid-stable-cyclic-Prop} 
are equivalent to their versions over $\CO$,
and Proposition \ref{proj-free-Frob-Prop} holds over $\CO$.
Following \cite{Punil}, the structure of the source algebras of nilpotent blocks 
as briefly reviewed at the beginning of Section \ref{nilpotent-Section} lifts
to $\CO$, hence so does  Lemma \ref{nilpotent-Lemma}. 
\end{Remark}

\begin{Remark} \label{double-coset-Remark}
If $g\in G$, and $ P, Q \leq G $ are  $p$-subgroups,  then the $P$-$Q$-biset 
$PgQ$ is transitive, and hence the 
$kP$-$kQ$-permutation bimodule $k[PgQ]$ is indecomposable with vertex 
$\{ (x, g^{-1}x^{-1} g):  x\in  P\cap gQg^{-1} \} $.
By the Krull-Schmidt Theorem,  $k[PgQ]$ is isomorphic to a direct summand
of $B$ regarded as $kP$-$kQ$-bimodule, for some (not necessarily unique) block $B$.
Thus the set $\Gamma=P\backslash G/Q$ of $P$-$Q$-double cosets in $G$
admits a partition $\Gamma=$
$\cup_{B} \Gamma_B$ such that for each block $B$ of $kG$ we have an
isomorphism of $kP$-$kQ$-bimodules $B\cong \oplus_{C\in \Gamma_B} k[C]$.
Here $k[C]$ is the $kP$-$kQ$-permutation bimodule with $k$-basis $C$.  
More generally, for $e$ an idempotent in $(kG)^P$ and
$f$ an idempotent in $(kG)^Q$
the space  $ekG f$ is a $kP$-$kQ$-bimodule summand of $kG$, and thus there
is a subset $\Gamma_{e,f}$ of $\Gamma$ such that 
$ekGf\cong \oplus_{C\in \Gamma_{e,f} } k[C]$.  Note that for any $P$-$Q$-double 
coset $C$ in $G$ the space 
$k\tenkP k[C] \tenkQ k$ is $1$-dimensional. By standard facts on tensor
products (see e. g. \cite[Proposition 2.9.3]{LiBookI} applied to $kGe\tenkP k$ and 
$kGf\tenkQ k)$  and using the symmetry of $kG$, which implies that $ekG$ and $kGe$ 
are dual to each other,   there is a linear isomorphism
\begin{Statement} \label{double-coset-isom}
$$\Hom_{kG}(kGe \tenkP k, kGf\tenkQ k) \cong k\tenkP ekG f\tenkQ k, $$
\end{Statement}
so in particular the dimension of this homomorphism space is equal to the number of
double cosets $|\Gamma_{e,f}|$ showing up in a decomposition of $ekGf$ as a
direct sum of $kP$-$kQ$-bimodules. The above isomorphism depends on the
choice of a symmetrising form of $kG$. In the situations mainly 
considered in this paper, $e$, $f$ are  either block idempotents of blocks of $kG$ 
or belong to points of $P$ or $Q$ on $kG$. 
\end{Remark}

\begin{Remark} \label{double-cosets-in-A}
Let $B$ a block of $kG$ with defect group $P$, and let
$i\in$ $B^P$ be a source idempotent. Set $A=iBi$.
The isomorphism \ref{double-coset-isom} in Remark \ref{double-coset-Remark}
applied with $e=f=i$ implies that  $\dim_k(\End_B(Bi\tenkP k))$ is equal to the 
number of $P$-$P$-orbits in a $P$-$P$-stable $k$-basis of $A$. 
Assuming that $k$ is algebraically closed,
Puig's finiteness conjecture predicts that for a fixed finite $p$-group $P$ there 
should only be finitely many isomorphism classes, as interior $P$-algebras,
of source algebras of blocks with defect groups isomorphic to $P$. If true,
Puig's conjecture would  imply  that the number of $P$-$P$-orbits in a stable 
basis of $A$ is  bounded in terms of $P$. More precisely, this would imply
that there are only finitely many isomorphism 
classes of algebras that arise as endomorphism algebras of source permutation 
modules of blocks with defect groups isomorphic to $P$. \end{Remark}

\begin{Remark}  
Let  $X$  be the set of  ordinary  irreducible characters of   $G$ of $p$-defect $0$. 
 For an  irreducible character   $\chi$  of $G$,  denote by $ \chi(1)' $ the $ p'$-part of  
 $\chi(1)$. Then the number of   $(S,S)$-double cosets  of  $G$   of  size  $ |S|^2 $
is greater than or equal to   $ \sum_{\chi \in X} (\chi (1)')^2$.
This is true simply because a defect zero block  $B$  is projective  (and hence free) as a  
$kS$-$kS$-bimodule. Therefore   the set  $\Gamma_B $   from  Remark
\ref{double-coset-Remark} 
consists   of  $\frac{\dim_k(B)}{|S|^2} $ double   cosets of the form  $k[Sg S] $, where  
$S \cap {^g\!S}=1 $. On the other hand, if $\chi $ is the unique  ordinary irreducible 
character  of $B$ (assuming $B$ is  split), then  $ \chi(1)^2  =  \dim_k B  $  and  
$\chi(1)=|S| \chi(1)'$. 
\end{Remark}

\begin{Remark}   
Suppose that $X$ is an indecomposable  $kG$-module summand of   $\Ind_S^G (k) $  
with vertex $Q \leq S$.  Then   $\Res^G_S (X)$    contains an 
indecomposable direct summand, say $Y$, with vertex $Q$ 
(see e.g. \cite[Proposition 5.1.6]{LiBookI}).  Now   $Y$ is also a 
summand of  $\Res^G_S  \Ind_S^G(k)   =$
$\oplus_{x\in S\backslash G/S} \Ind_{{^x\!{S}} \cap  S}^S(k) $. 
By Green's indecomposability theorem,  for any $ \in G$,   
$\Ind_{{^x\!{S}} \cap S }^S(k) $  is  indecomposable with  vertex $ {^x\!{S}} \cap S  $.  So  
$Q$ is $S$-conjugate to  
$ {^x\!{S}} \cap S  $ for some $x \in G$.  Thus,   by \cite[Lemma 1]{Alp87} (see 
Theorem \ref{Bi-weight-thm})  and  \cite[Theorem 1.2]{Rob88} (see Theorem 
\ref{princ-block-thm})  if $Q  \leq  S$ is a $p$-subgroup  of $G$  which contributes  
a $G$-weight  or if $Q$  essential in the fusion system $\CF_S(G)$,  then $Q$ is  
the intersection  of two Sylow subgroups of  $G$.   This generalises  Green's  
result that block  defect groups  are Sylow intersections.
\end{Remark}

\bigskip\noindent
{\bf Acknowledgements.} 
The second author  acknowledges support from EPSRC grant EP/X035328/1.


\end{document}